\documentclass[12pt,a4paper,english]{article}
\usepackage[T1]{fontenc}
\usepackage[latin9]{inputenc}
\usepackage{color}
\usepackage{amsmath}
\usepackage{amssymb}

\makeatletter

\usepackage{amsthm}\usepackage{fullpage}\usepackage{makeidx}\usepackage{amsfonts}\usepackage{latexsym}\makeindex

\def\Z{{\mathbb{Z}}}

\def \<{\langle}
\def \>{\rangle}
\theoremstyle{definition}
\newtheorem{lemma}{Lemma}[section]
\newtheorem{theorem}[lemma]{Theorem}

\newtheorem{proposition}[lemma]{Proposition}
\newtheorem{definition}[lemma]{Definition}

\newtheorem{remark}[lemma]{Remark}
\usepackage{times}

\usepackage{babel}

\usepackage{babel}

\usepackage[blocks]{authblk}% The option is for block layout
\title{2-permutations of  lattice vertex operator algebras: Higher rank
}
\author{Chongying Dong\footnote{Supported by NSF grant DMS-1404741 and China NSF grant 11371261}}
\affil{Department of Mathematics, University of
California, Santa Cruz, CA 95064 USA}
\author{Feng Xu\footnote{Partially supported by China NSF grant 11471064}}
\author{Nina Yu}
\affil{Department of Mathematics, University of California, Riverside, CA 92521 USA}

\usepackage{babel}

\usepackage{babel}

\usepackage{babel}

\makeatother

\usepackage{babel}
\begin{document}
\maketitle
\begin{abstract}
The fusion rules
of the 2-permutation orbifold of an arbitrary lattice vertex operator
algebra are determined by using the theory of quantum dimension.
\end{abstract}

\section{Introduction}

This paper is a continuation of our investigation on 2-permutation
of lattice vertex operator algebras \cite{DXY}. In particular, the
quantum dimensions of irreducible modules and the fusion rules are
determined. If the rank of the lattice is one, these results have
been obtained previously in \cite{DXY}.

Let $V$ be a vertex operator algebra and $n$ a fixed positive integer
and consider the tensor product vertex operator algebra $V^{\otimes n}$
\cite{FHL}. Then the symmetric group $S_{n}$ acts naturally on $V^{\otimes n}$
as automorphisms. The permutation orbifold theory has been studied
extensively in physics \cite{KS,FKS,BHS,Ba}. Conformal nets approach
to permutation orbifolds have been given in \cite{KLX}. Twisted sectors
of permutation orbifolds of tensor products of an arbitrary vertex
operator algebra have been constructed in \cite{BDM}. The $C_{2}$-cofiniteness
of permutation orbifolds and general cyclic orbifolds have been studied
in \cite{A3,A4,M}. But the representation theory such as rationality,
classification of irreducible modules, and fusion rules for the fixed
point vertex operator algebra $(V^{\otimes n})^{G}$ for any $n$
and any subgroup $G$ of $S_{n}$ have not been investigated much.

As a starting point, we studied representations of 2-permutation orbifold
model of rank one lattice vertex operator algebras in \cite{DXY}.
In this paper, we complete the study of 2-permutation orbifold model
of lattice vertex operator algebras $V_{L}$ for any positive definite
even lattice $L.$ Similar to rank one case, the permutation orbifold
model $(V_{L}\otimes V_{L})^{\Z_{2}}$ can be realized as a simple
current extension of the rational vertex operator algebra $V_{\sqrt{2}L}\otimes V_{\sqrt{2}L}^{+}$.
It follows from \cite{Y,HKL} that $(V_{L}\otimes V_{L})^{\Z_{2}}$
is rational. According to \cite{DRX}, every irreducible $(V_{L}\otimes V_{L})^{\Z_{2}}$-module
occurs in an irreducible $g$-twisted $V_{L}\otimes V_{L}$-module. So the classification of irreducible
$(V_{L}\otimes V_{L})^{\Z_{2}}$-modules is known. But this classification result does not suggest
how to compute the fusion rules among the irreducible modules.
The main idea is to use the general theory of simple
current extension of a rational vertex operator algebra and representations
of $V_{L}$ and $V_{L}^{+}$ to study the representations of $\left(V_{L}\otimes V_{L}\right)^{\mathbb{Z}_{2}}$.
We decompose each irreducible $V_{L}\otimes V_{L}$-module
into a direct sum of irreducible $(V_{L}\otimes V_{L})^{\mathbb{Z}_{2}}$-modules
by using the fusion rules for both vertex operator algebras $V_{\sqrt{2}L}$
and $V_{\sqrt{2}L}^{+}$ \cite{DL1,A1,ADL}. This decomposition is crucial in computing
the fusion rules. We emphasize that the theory of quantum dimensions introduced and studied in \cite{DJX, DRX} plays an essential role in computing the fusion rules. It is not clear to us how to achieve this  without using the quantum dimensions. The fusion rules in conformal nets for any 2-permutation
models were computed by using the $S$-matrix \cite{KLX}.

We should mention that the constructions of $g$-twisted modules for lattice vertex operator algebra $V_L$ where
$g$ is automorphism of finite order induced from an isometry of $L$ were already given in
 \cite{FLM1, FLM2, L, DL2}.  In the case $g$ is of order $2,$ the irreducible modules of $V_L^{\<g\>}$
 have been classified recently in \cite{BE}. An equivalence of two constructions of permutation-twisted modules
 for lattice vertex operator algebras in \cite{FLM1, L} and \cite{BDM} was given in \cite{BHL}.

The paper is organized as follows: $ $\S2 and \S3 are preliminaries
on the vertex operator algebras theory. In these sections we give
some basic notions that appear in this paper and recall the constructions
of the lattice type vertex operator algebras $V_{L}$ and $V_{L}^{+}$
and their (twisted) modules. In \S4 we study $\left(V_{L}\otimes V_{L}\right)^{\mathbb{Z}_{2}}$,
the 2-cyclic permutation orbifold models for rank $d$ lattice vertex
operator algebras. In particular, we decompose each irreducible $V_{L}\otimes V_{L}$-module into a direct sum of irreducible $\left(V_{L}\otimes V_{L}\right)^{\mathbb{Z}_{2}}$-modules.
The quantum dimensions of all irreducible modules
of $\left(V_{L}\otimes V_{L}\right)^{\mathbb{Z}_{2}}$ are obtained explicitly
in \S5. Finally, we apply results from the previous sections to determine
all fusion products in \S6.

\section{Preliminaries}

Let $\left(V,Y,\mathbf{1},\omega\right)$ be a vertex operator algebra
\cite{Bo,FLM2} and $g$ an automorphism of vertex operator algebra
$V$ of order $T$. Denote the decomposition of $V$ into eigenspaces
of $g$ as:

\[
V=\oplus_{r=0}^{T-1}V^{r},\ V^{r}=\left\{ v\in V|gv=e^{2\pi ir/T}v\right\} .
\]

Here are the definitions of weak, admissible, ordinary $g$-twisted
$V$-modules \cite{DLM3}. \begin{definition}A \emph{weak $g$-twisted
$V$-module} $M$ is a vector space with a linear map
\[
Y_{M}:V\to\left(\text{End}M\right)\{z\}
\]

\[
v\mapsto Y_{M}\left(v,z\right)=\sum_{n\in\mathbb{Q}}v_{n}z^{-n-1}\ \left(v_{n}\in\mbox{End}M\right)
\]

which satisfies the following: for all $0\le r\le T-1$, $u\in V^{r}$,
$v\in V$, $w\in M$,

(1) $u_{l}w=0$ if $l$ is sufficiently large,

(2) $Y_{M}\left(u,z\right)=\sum_{n\in\frac{r}{T}+\mathbb{Z}}u_{n}z^{-n-1},$

(3) $Y_{M}\left(\mathbf{1},z\right)=Id_{M},$

(4) (Twisted Jacobi identity)

\[
z_{0}^{-1}\text{\ensuremath{\delta}}\left(\frac{z_{1}-z_{2}}{z_{0}}\right)Y_{M}\left(u,z_{1}\right)Y_{M}\left(v,z_{2}\right)-z_{0}^{-1}\delta\left(\frac{z_{2}-z_{1}}{-z_{0}}\right)Y_{M}\left(v,z_{2}\right)Y_{M}\left(u,z_{1}\right)
\]

\[
z_{2}^{-1}\left(\frac{z_{1}-z_{0}}{z_{2}}\right)^{-r/T}\delta\left(\frac{z_{1}-z_{0}}{z_{2}}\right)Y_{M}\left(Y\left(u,z_{0}\right)v,z_{2}\right),
\]
where $\delta\left(z\right)=\sum_{n\in\mathbb{Z}}z^{n}$. \end{definition}

\begin{definition}An \emph{admissible $g$-twisted $V$-module} $M=\oplus_{n\in\frac{1}{T}\mathbb{Z}_{+}}M\left(n\right)$
is a $\frac{1}{T}\mathbb{Z}_{+}$-graded weak $g$-twisted module
such that $u_{m}M\left(n\right)\subset M\left(\mbox{wt}u-m-1+n\right)$
for homogeneous $u\in V$ and $m,n\in\frac{1}{T}\mathbb{Z}.$ $ $
\end{definition}

\begin{definition}

A (ordinary) $g$-\emph{twisted $V$-module} is a weak $g$-twisted
$V$-module $M$ which carries a $\mathbb{C}$-grading induced by
the spectrum of $L(0)$, where $L(0)$ is the component operator of
$Y(\omega,z)=\sum_{n\in\mathbb{Z}}L(n)z^{-n-2}.$ That is, we have
$M=\bigoplus_{\lambda\in\mathbb{C}}M_{\lambda},$ where $M_{\lambda}=\{w\in M|L(0)w=\lambda w\}$.
Moreover, $\dim M_{\lambda}$ is finite and for fixed $\lambda,$
$M_{\frac{n}{T}+\lambda}=0$ for all small enough integers $n.$ A
vector $w\in M_{\lambda}$ is called a weight vector of weight $\lambda$,
and write $\lambda=\mbox{wt}w$.

\end{definition}

\begin{remark} \label{g action}If $g=Id_{V}$ we have the notions
of weak, ordinary and admissible $V$-modules.

Note that the cyclic group $\left\langle g\right\rangle $ generated
by $g$ acts on any admissible $g$-twisted $V$-module $M$ such
that $g|_{M(n)}=e^{-2\pi in}$ for $n\in\frac{1}{T}\Z$ and $gY_{M}(v,z)g^{-1}=Y_{M}(gv,z)$
for all $v\in V.$ In particular, $M^{r}=\oplus_{n\in\Z}M(\frac{r}{T}+n)$
is an admissible $V^{\left\langle g\right\rangle }$-module for $r=0,...,T-1.$
Moreover, if $M$ is irreducible then each $M^{r}$ is irreducible
admissible $V^{\left\langle g\right\rangle }$-module \cite{DY,MT,DRX}.

\end{remark}

\begin{definition}A vertex operator algebra $V$ is called \emph{$g$-rational}
if the admissible $g$-twisted module category is semisimple. $V$
is called \emph{rational} if $V$ is $1$-rational. \end{definition}

\begin{definition} A vertex operator algebra $V$ is said to be \emph{$C_{2}$-cofinite}
if $V/C_{2}(V)$ is finite dimensional, where $C_{2}(V)=\langle v_{-2}u|v,u\in V\rangle.$
\end{definition}

\begin{remark} If vertex operator algebra $V$ is rational or $C_{2}$-cofinite,
then $V$ has only finitely many irreducible admissible modules up
to isomorphism and each irreducible admissible module is ordinary
\cite{DLM3,Li}. \end{remark}

Now we consider the tensor product vertex algebras and the tensor
product modules for tensor product vertex operator algebras. The tensor
product of vertex operator algebras $\left(V^{1},Y^{1},1,\omega^{1}\right)$
and $\left(V^{2},Y^{2},1,\omega^{2}\right)$ is constructed on the
tensor product vector space $V=V^{1}\otimes V^{2}$ where the vertex
operator $Y\left(\cdot,z\right)$ is defined by $Y\left(v^{1}\otimes v^{2},z\right)=Y\left(v^{1},z\right)\otimes Y\left(v^{2},z\right)$
for $v^{i}\in V^{i}$, $i=1,2$, the vacuum vector is $\mathbf{1}=1\otimes1$
and the Virasoro element is $\omega=\omega^{1}\otimes\omega^{2}$.
Then $\left(V,Y,\mathbf{1},\omega\right)$ is a vertex operator algebra
\cite{FHL,LL}. Let $W^{i}$ be an admissible $V^{i}$-module for
$i=1,2$. We may construct the tensor product admissible module $W^{1}\otimes W^{2}$
for the tensor product vertex operator algebra $V^{1}\otimes V^{2}$
by $Y\left(v^{1}\otimes v^{2},z\right)=Y\left(v^{1},z\right)\otimes Y\left(v^{2},z\right)$.
Then $W^{1}\otimes W^{2}$ is an admissible $V^{1}\otimes V^{2}$-module.
We have the following result about tensor product modules \cite{DMZ,FHL}:

\begin{theorem} Let $V^{1},V^{2}$ be rational vertex operator algebras,
then $V^{1}\otimes V^{2}$ is rational and any irreducible $V^{1}\otimes V^{2}$-module
is a tensor product $W^{1}\otimes W^{2}$ for some irreducible $V^{i}$-module
$W^{i}$ and $i=1,2.$

\end{theorem}

Let $M=\bigoplus_{n\in\frac{1}{T}\mathbb{Z}_{+}}M(n)$ be an admissible
$g$-twisted $V$-module, the contragredient module $M'$ is defined
as follows:
\[
M'=\bigoplus_{n\in\frac{1}{T}\mathbb{Z}_{+}}M(n)^{*},
\]
where $M(n)^{*}=\mbox{Hom}_{\mathbb{C}}(M(n),\mathbb{C}).$ The vertex
operator $Y_{M'}(v,z)$ is defined for $v\in V$ via
\begin{eqnarray*}
\langle Y_{M'}(v,z)f,u\rangle= & \langle f,Y_{M}(e^{zL(1)}(-z^{-2})^{L(0)}v,z^{-1})u\rangle
\end{eqnarray*}
where $\langle f,w\rangle=f(w)$ is the natural paring $M'\times M\to\mathbb{C}.$
Then $M'$ is an admissible $g^{-1}$-twisted $V$-module \cite{X}.
A $V$-module $M$ is said to be \emph{ self dual} if $M$ and $M'$ are isomorphic
$V$-modules.

We now recall the notion of intertwining operators and fusion rules
\cite{FHL}:
\begin{definition} Let $(V,\ Y)$ be a vertex operator
algebra and let $(W^{1},\ Y^{1}),\ (W^{2},\ Y^{2})$ and $(W^{3},\ Y^{3})$
be $V$-modules. An intertwining operator of type $\left(\begin{array}{c}
W^{1}\\
W^{2\ }W^{3}
\end{array}\right)$ is a linear map
\[
I(\cdot,\ z):\ W^{2}\to\text{\ensuremath{\mbox{Hom}(W^{3},\ W^{1})\{z\}}}
\]

\[
u\to I(u,\ z)=\sum_{n\in\mathbb{Q}}u_{n}z^{-n-1}
\]
satisfying:

(1) for any $u\in W^{2}$ and $v\in W^{3}$, $u_{n}v=0$ for $n$
sufficiently large;

(2) $I(L_{-1}v,\ z)=(\frac{d}{dz})I(v,\ z)$;

(3) (Jacobi Identity) for any $u\in V,\ v\in W^{2}$

\[
z_{0}^{-1}\delta\left(\frac{z_{1}-z_{2}}{z_{0}}\right)Y^{1}(u,\ z_{1})I(v,\ z_{2})-z_{0}^{-1}\delta\left(\frac{-z_{2}+z_{1}}{z_{0}}\right)I(v,\ z_{2})Y^{3}(u,\ z_{1})
\]
\[
=z_{2}^{-1}\left(\frac{z_{1}-z_{0}}{z_{2}}\right)I(Y^{2}(u,\ z_{0})v,\ z_{2}).
\]

We denote the space of all intertwining operators of type $\left(\begin{array}{c}
W^{1}\\
W^{2}\ W^{3}
\end{array}\right)$ by $I_{V}\left(\begin{array}{c}
W^{1}\\
W^{2}\ W^{3}
\end{array}\right).$ Let $N_{W^{2},\ W^{3}}^{W^{1}}=N_{V}\left(\begin{array}{c}
W^{1}\\
W^{2}\ W^{3}
\end{array}\right)=\dim I_{V}\left(\begin{array}{c}
W^{1}\\
W^{2}\ W^{3}
\end{array}\right)$. These integers $N_{W^{2},\ W^{3}}^{W^{1}}$ are usually called the
\emph{fusion rules}. \end{definition}

\begin{definition} Let $V$ be a vertex operator algebra, and $W^{1},$
$W^{2}$ be two $V$-modules. A module $(W,I)$, where $I\in I_{V}\left(\begin{array}{c}
\ \ W\ \\
W^{1}\ \ W^{2}
\end{array}\right),$ is called a \emph{fusion product} of $W^{1}$ and $W^{2}$ if for any $V$-module
$M$ and $\mathcal{Y}\in I_{V}\left(\begin{array}{c}
\ \ M\ \\
W^{1}\ \ W^{2}
\end{array}\right),$ there is a unique $V$-module homomorphism $f:W\rightarrow M,$ such
that $\mathcal{Y}=f\circ I.$ As usual, we denote $(W,I)$ by $W^{1}\boxtimes_{V}W^{2}.$
\end{definition}

It is well known that if $V$ is rational, then the fusion product
exists. We shall often consider the fusion product
\[
W^{1}\boxtimes_{V}W^{2}=\sum_{W}N_{W^{1},\ W^{2}}^{W}W
\]
where $W$ runs over the set of equivalence classes of irreducible
$V$-modules.

The fusion rules satisfy the following symmetry \cite{FHL}.
\begin{proposition}\label{fusion rule symmmetry property}
Let $W^{i}$ $\left(i=1,2,3\right)$ be $V$-modules. Then

\[
\dim I_{V}\left(_{W^{1}W^{2}}^{\ \ W^{3}}\right)=\dim I_{V}\left(_{W^{2}W^{1}}^{\ \ W^{3}}\right),\dim I_{V}\left(_{W^{1}W^{2}}^{\ \ W^{3}}\right)=\dim I_{V}\left(_{W^{1}\left(\ W^{3}\right)'}^{\ \ \left(W^{2}\right)'}\right).
\]

\end{proposition}

\begin{definition} Let $V$ be a simple vertex operator algebra.
A simple $V$-module $M$ is called \emph{a simple current} if for
any irreducible $V$-module $W$, $W\boxtimes M$ exists and is also
a simple $V$-module. \end{definition}

Let $D$ be a finite abelian group and assume that we have a set of
irreducible simple current $V^{0}$-modules $\left\{ V^{\alpha}|\alpha\in D\right\} $
indexed by $D$. The following definition is from \cite{Y}.

\begin{definition} An extension $V_{D}=\oplus_{\alpha\in D}V^{\alpha}$
of $V^{0}$ is called a $D$-\emph{graded simple current extension
}if $V_{D}$ carries a structure of a simple vertex operator algebra
such that $Y\left(u^{\alpha},z\right)u^{\beta}\in V^{\alpha+\beta}\left(\left(z\right)\right)$
for any $u^{\alpha}\in V^{\alpha}$ and $u^{\beta}\in V^{\beta}.$
\end{definition}

\section{Vertex Operator algebra $V_{L}$ and $V_{L}^{+}$}

We first review the construction of the vertex operator algebra $V_{L}$
associated with a positive definite even lattice $L$ \cite{Bo,FLM2}.

We are working in the setting of \cite{DL1,FLM2}. Let $L$ be a positive
definite even lattice with bilinear form $\left\langle \cdot,\cdot\right\rangle $
and $L^{\circ}$ its dual lattice in $\mathfrak{h}=\mathbb{C}\otimes_{\mathbb{Z}}L$.
Let $\left\{ \lambda_{0}=0,\lambda_{1},\lambda_{2},\cdots\right\} $
be a complete set of coset representatives of $L$ in $L^{\circ}.$
Then the lattice vertex operator algebra $V_{L}$ is rational and
$V_{\lambda_{i}+L}\mbox{}$ are the irreducible $V_{L}$-modules \cite{Bo,FLM2,D1,DLM2}.

Now assume $M$ is positive definite even lattice such that $\left\langle \alpha,\beta\right\rangle \in2\mathbb{Z}$
for $\alpha,\beta\in M$. In this case, $V_{M+\lambda}=S\left(\mathfrak{h}\otimes t^{-1}\mathbb{C}\left[t^{-1}\right]\right)\otimes\mathbb{C}\left[\lambda+M\right]$
for any $\lambda\in M^{\circ}$ where $S\left(\cdot\right)$ is the
symmetric algebra and $\mathbb{C}\left[\lambda+M\right]=\sum_{\alpha\in M}\mathbb{C} e_{\lambda+\alpha}$
is the subspace of the group algebra $\mathbb{C}\left[M^{\circ}\right]$
corresponds to $\lambda+M$. Then $V_{M}$ has a canonical automorphism
$\theta$ of order 2 induced from $-1$ isometry of $M$. In fact,
we can define a linear map $\theta$ from $V_{\lambda+M}$ to $V_{-\lambda+M}$
for any $\lambda\in M^{\circ}$ such that $\theta Y_{\lambda+M}\left(u,z\right)\theta^{-1}=Y_{-\lambda+M}\left(\theta u,z\right)$
for any $u\in V_{M}$ where $Y_{\lambda+M}$ defines a $V_{M}$-module
structure on $V_{\lambda+M}$ \cite{AD}. Clearly, if $2\lambda\in M$,
$\theta$ is an endormorphism from $V_{\lambda+M}$ to $V_{\lambda+M}$.
For such $\lambda$ we denote eigenspace of $\theta$ with eigenvalue
$\pm1$ in $V_{\lambda+M}$ by $V_{\lambda+M}^{\pm}$.

We now turn our attention to the construction of $\theta$-twisted
$V_{M}$-modules \cite{FLM1,FLM2,L,DL2}. Note that $M/2M$ is an
abelian group of order $2^{d}$ where $d$ is the rank of $M$. Then
$M/2M$ has exactly $2^{d}$ inequivalent irreducible modules $T_{\chi}$
where $\chi$ is irreducible character of $M/2M$. It was proved in
\cite{FLM2} that $V_{M}^{T_{\chi}}=S\left(\mathfrak{\mathfrak{h}\otimes}\left(t^{-\frac{1}{2}}\right)\mathbb{C}\left[t^{-1}\right]\right)\otimes T_{\chi}$
is an irreducible $\theta$-twisted module. According to Remark \ref{g action},
$\theta$ acts on $V_{M}^{T_{\chi}}$. Again we denote eigenspace
of $\theta$ with eigenvalue $\pm1$ by $V_{M}^{T_{\chi,}\pm}$. Moreover,
$V_{M}$ is $\theta$-rational and $\left\{ V_{M}^{T_{\chi}}|\chi\right\} $
gives a complete list of inequivalent irreducible $\theta$-twisted
$V_{M}$-modules \cite{D2}.

We have the following classification of the irreducible $V_{L}^{+}$-modules
\cite{AD,DN}:

\begin{theorem} Let $M$ be a positive definite even lattice and
$\left\{ \lambda_{i}\right\} $ be a set of coset representatives
of $M$ in $M^{\circ}$. Then any irreducible $V_{M}^{+}$-module
is isomorphic to one of the following:
\[
V_{\lambda_{i}+M}\left(2\lambda_{i}\notin M\right),V_{\lambda_{i}+M}^{\pm}\left(2\lambda_{i}\in M\right),V_{M}^{T_{\chi},\pm}.
\]
Furthermore, $V_{\lambda_{i}+M}\cong V_{\lambda_{j}+M}$ if and only
if $\lambda_{i}\pm\lambda_{j}\in M$.

\end{theorem}

From now on, we fix a rank $d$ lattice $L=\mathbb{Z}\alpha_{1}+\cdots+\mathbb{Z}\alpha_{d}$
with positive definite symmetric non-degenerate bilinear form $\left\langle \cdot,\cdot\right\rangle $.
Let $M=\sqrt{2}L$. Later we will see that the 2-permutation orbifold
model we study is closely related to the rational vertex operator
algebras $V_{\sqrt{2}L}$ and $V_{\sqrt{2}L}^{+}$. We now consider
the fusion rules for the vertex operator algebras $V_{\sqrt{2}L}$
and $V_{\sqrt{2}L}^{+}$.

First we notice that the dual lattice of $\sqrt{2}L$ can be written
by $\left(\sqrt{2}L\right)^{\circ}=\left\{ \frac{\lambda}{\sqrt{2}}|\lambda\in L^{\circ}\right\} $.
Thus fusion rules for irreducible $V_{\sqrt{2}L}$-modules are given
by the following \cite{DL1}:

\begin{proposition}\label{Fusion-V_L} $N_{_{V_{\sqrt{2}L}}}\left(_{V_{\frac{\lambda}{\sqrt{2}}+\sqrt{2}L}\ V_{\frac{\mu}{\sqrt{2}}+\sqrt{2}L}}^{V_{\frac{\gamma}{\sqrt{2}}+\sqrt{2}L}}\right)=\delta_{\frac{\lambda+\mu}{\sqrt 2}+\sqrt{2}L,\frac{\gamma}{\sqrt 2}+\sqrt{2}L}$
for $\lambda,\mu$ and $\gamma\in L^{\circ}$.

\end{proposition}

The fusion rules for $V_{L}^{+}$ for any $L$ was obtain in \cite{ADL}. For this purpose, we need to
identify the contragredient modules of the irreducible $V_{\sqrt{2}L}^{+}$-modules
first (see Proposition 3.7 of \cite{ADL}).

\begin{proposition}\label{self-dual} Every irreducible $V_{\sqrt{2}L}^{+}$-module is self dual.
\end{proposition}

\begin{remark} \label{def of chi} For any $\lambda \in L^\circ$, we define a character $\chi_{\lambda}$ so that  $\chi_{\lambda}\left(\sqrt{2}\alpha_{i}\right)=\left(-1\right)^{\frac{\left\langle \alpha_{i},\alpha_{i}\right\rangle }{2}+\left\langle \lambda,\alpha_{i}\right\rangle }$
for $1\le i\le d$. Then $\chi_{\mu}=\chi_{\lambda}$ if and only if $\lambda-\mu\in 2L^{\circ}.$ Thus
 $\left\{ \chi_{\lambda}|\lambda\in L^{\circ}/2L^{\circ}\right\} $
gives all different characters.

\end{remark}

Recall  the number $\pi_{\lambda,\mu}=e^{\langle \lambda, \mu\rangle \pi i}$ for $\lambda,\mu\in (\sqrt 2 L)^{\circ} \cite{ADL}.$  For any character $\chi$ of ${\sqrt 2 L}/{2\sqrt 2 L}$, $c_\chi$ was defined in  \cite{ADL} and we note that here  for any $\alpha\in L$ we have:

\[
c_{\chi}\left(\frac{\alpha}{\sqrt{2}}\right)=\left(-1\right)^{\left\langle \alpha,\alpha\right\rangle }\chi\left(\sqrt{2}\alpha\right)=\chi\left(\sqrt{2}\alpha\right).
\]

 For any $\mu\in L^{\circ}, \alpha\in L$, the character $\chi^{\left(\frac{\mu}{\sqrt{2}}\right)}$
is defined in \cite{ADL} by
\[
\chi_{\lambda}^{\left(\frac{\mu}{\sqrt{2}}\right)}\left(\sqrt 2 \alpha \right)=(-1)^{\langle \alpha,\mu\rangle }\chi_{\lambda}\left(\sqrt 2 \alpha \right)
\]
and $T_{\chi_\lambda^{\left(\frac{\mu}{\sqrt{2}}\right)}}$ is denoted by $T_{\chi_\lambda}^{\left(\frac{\mu}{\sqrt{2}}\right)}$. By the definition of $\chi_\lambda$ in Remark \ref{def of chi}, it is easy to  check that for $\alpha\in L$, $\lambda\in L^\circ$,
$\chi_{\lambda}^{\left(\frac{\alpha}{\sqrt{2}}\right)}\left(\sqrt{2}\alpha_{i}\right)=\chi_{\lambda+\alpha}$$\left(\sqrt{2}\alpha_{i}\right)$
for $1\le i\le d$ and hence $T_{\chi_{\lambda}}^{\left(\frac{\alpha}{\sqrt{2}}\right)}=T_{\chi_{\lambda+\alpha}}$
for any $\lambda\in L^\circ$ and $\alpha\in L$.

A triple $\left(\lambda,\mu,\gamma\right)\subset L^{\circ}$ is said to be an admissible
triple modulo $L$ if $p\lambda+q\mu+r\gamma\in L$ for some $p,q,r\in\left\{ \pm1\right\} $.
Now we are ready to list fusion rules of  irreducible $V_{\sqrt{2}L}^{+}$-modules
\cite{ADL}:

%%%                           Fusion-V_L+                    %%%%%%%%%%%%%%%%%%%%%
\begin{proposition}\label{Fusion-V_L+} % Preview source code from paragraph 0 to 24

Let $L$ be the rank $d$ lattice as before. For any irreducible $V_{\sqrt{2}L}^{+}$-modules
$M^{i}$ ($i=1,2,3$), the fusion rule of type $\left(_{M^{1}M^{2}}^{\ \ M^{3}}\right)$
is $1$ if and only if $M^{i}$ ( $i=1,2,3$ ) satisfy one of the
following conditions:

(i) $M^{1}=V_{\frac{\lambda}{\sqrt{2}}+\sqrt{2}L}$ for $\lambda\in L^{\circ}$
such that $\lambda\notin L$ and $\left(M^{2},\ M^{3}\right)$ is
one of the following pairs:

$\left(V_{\frac{\mu}{\sqrt{2}}+\sqrt{2}L},V_{\frac{\gamma}{\sqrt{2}}+\sqrt{2}L}\right)$
for $\mu,\gamma\in L^{\circ}$ such that $\mu,\gamma\not\in L$ and
$\left(\frac{\lambda}{\sqrt{2}},\frac{\mu}{\sqrt{2}},\frac{\gamma}{\sqrt{2}}\right)$
is an admissible triple modulo $\sqrt{2}L$,

$\left(V_{\frac{\mu}{\sqrt{2}}+\sqrt{2}L}^{\pm},V_{\frac{\gamma}{\sqrt{2}}+\sqrt{2}L}\right)$,
$\left(V_{\frac{\gamma}{\sqrt{2}}+\sqrt{2}L},V_{\frac{\mu}{\sqrt{2}}+\sqrt{2}L}^{\pm}\right)$for
$\mu\in L,\gamma\in L^{\circ}$ and $\left(\frac{\lambda}{\sqrt{2}},\frac{\mu}{\sqrt{2}},\frac{\gamma}{\sqrt{2}}\right)$
is an admissible triple modulo $\sqrt{2}L$,

$\left(V_{\sqrt{2}L}^{T_{\chi_{\mu},}\pm},V_{\sqrt{2}L}^{T_{\chi_{\mu}}^{\left(\frac{\lambda}{\sqrt{2}}\right)},\pm}\right)$,
$\left(V_{\sqrt{2}L}^{T_{\chi_{\mu},}\pm},V_{\sqrt{2}L}^{T_{\chi_{\mu}}^{\left(\frac{\lambda}{\sqrt{2}}\right)},\mp}\right)$for
any irreducible $\sqrt{2}L/2\sqrt{2}L$-module $T_{\chi_{\mu}}.$

(ii) $M^{1}=V_{\frac{\lambda}{\sqrt{2}}+\sqrt{2}L}^{+}$ for $\lambda\in L$
and $\left(M^{2},\ M^{3}\right)$ is one of the following pairs:

$\left(V_{\frac{\mu}{\sqrt{2}}+\sqrt{2}L},V_{\frac{\gamma}{\sqrt{2}}+\sqrt{2}L}\right)$
for $\mu,\gamma\in L^{\circ}$ such that $\mu\not\in L$ and $\left(\frac{\lambda}{\sqrt{2}},\frac{\mu}{\sqrt{2}},\frac{\gamma}{\sqrt{2}}\right)$
is an admissible triple modulo $\sqrt{2}L$,

$\left(V_{\frac{\mu}{\sqrt{2}}+\sqrt{2}L}^{\pm},V_{\frac{\gamma}{\sqrt{2}}+\sqrt{2}L}^{\pm}\right)$
for $\mu\in L$, $\gamma\in L^{\circ}$ such that $\pi_{\frac{\lambda}{\sqrt{2}},\sqrt{2}\mu}=1$
and $\left(\frac{\lambda}{\sqrt{2}},\frac{\mu}{\sqrt{2}},\frac{\gamma}{\sqrt{2}}\right)$
is an admissible triple modulo $\sqrt{2}L$,

$\left(V_{\frac{\mu}{\sqrt{2}}+\sqrt{2}L}^{\pm},V_{\frac{\gamma}{\sqrt{2}}+\sqrt{2}L}^{\mp}\right)$
for $\mu\in L$, $\gamma\in L^{\circ}$ such that $\pi_{\frac{\lambda}{\sqrt{2}},\sqrt{2}\mu}=-1$
and $\left(\frac{\lambda}{\sqrt{2}},\frac{\mu}{\sqrt{2}},\frac{\gamma}{\sqrt{2}}\right)$
is an admissible triple modulo $\sqrt{2}L$,

$\left(V_{\sqrt{2}L}^{T_{\chi_{\mu},}\pm},V_{\sqrt{2}L}^{T_{\chi_{\mu+\lambda}},\pm}\right)$,
$\left(V_{\sqrt{2}L}^{T_{\chi_{\mu+\lambda}},\pm},V_{\sqrt{2}L}^{T_{\chi_{\mu},}\pm}\right)$
for any irreducible $\sqrt{2}L/2\sqrt{2}L$-module $T_{\chi_{\mu}}$
such that $\chi_{\mu}\left(\sqrt{2}\lambda\right)=1,$

$\left(V_{\sqrt{2}L}^{T_{\chi_{\mu},}\pm},V_{\sqrt{2}L}^{T_{\chi_{\mu+\lambda}},\mp}\right)$,
$\left(V_{\sqrt{2}L}^{T_{\chi_{\mu+\lambda}},\mp},V_{\sqrt{2}L}^{T_{\chi_{\mu},}\pm}\right)$
for any irreducible $\sqrt{2}L/2\sqrt{2}L$-module $T_{\chi_{\mu}}$
such that $\chi_{\mu}\left(\sqrt{2}\lambda\right)=-1.$

(iii) $M^{1}=V_{\frac{\lambda}{\sqrt{2}}+\sqrt{2}L}^{-}$ for $\lambda\in L$
and $\left(M^{2},\ M^{3}\right)$ is one of the following pairs:

$\left(V_{\frac{\mu}{\sqrt{2}}+\sqrt{2}L},V_{\frac{\gamma}{\sqrt{2}}+\sqrt{2}L}\right)$
for $\mu,\gamma\in L^{\circ}$ such that $\mu\not\in L$ and $\left(\frac{\lambda}{\sqrt{2}},\frac{\mu}{\sqrt{2}},\frac{\gamma}{\sqrt{2}}\right)$
is an admissible triple modulo $\sqrt{2}L$,

$\left(V_{\frac{\mu}{\sqrt{2}}+\sqrt{2}L}^{\pm},V_{\frac{\gamma}{\sqrt{2}}+\sqrt{2}L}^{\mp}\right)$
for $\mu\in L$, $\gamma\in L^{\circ}$ such that $\pi_{\frac{\lambda}{\sqrt{2}},\sqrt{2}\mu}=1$
and $\left(\frac{\lambda}{\sqrt{2}},\frac{\mu}{\sqrt{2}},\frac{\gamma}{\sqrt{2}}\right)$
is an admissible triple modulo $\sqrt{2}L$,

$\left(V_{\frac{\mu}{\sqrt{2}}+\sqrt{2}L}^{\pm},V_{\frac{\gamma}{\sqrt{2}}+\sqrt{2}L}^{\pm}\right)$
for $\mu\in L$, $\gamma\in L^{\circ}$ such that $\pi_{\frac{\lambda}{\sqrt{2}},\sqrt{2}\mu}=-1$
and $\left(\frac{\lambda}{\sqrt{2}},\frac{\mu}{\sqrt{2}},\frac{\gamma}{\sqrt{2}}\right)$
is an admissible triple modulo $\sqrt{2}L$,

$\left(V_{\sqrt{2}L}^{T_{\chi_{\mu},}\pm},V_{\sqrt{2}L}^{T_{\chi_{\mu+\lambda}},\mp}\right)$,
$\left(V_{\sqrt{2}L}^{T_{\chi_{\mu+\lambda}},\mp},V_{\sqrt{2}L}^{T_{\chi_{\mu},}\pm}\right)$for
any irreducible $\sqrt{2}L/2\sqrt{2}L$-module $T_{\chi_{\mu}}$ such
that $\chi_{\mu}\left(\sqrt{2}\lambda\right)=1,$

$\left(V_{\sqrt{2}L}^{T_{\chi_{\mu},}\pm},V_{\sqrt{2}L}^{T_{\chi_{\mu+\lambda}},\pm}\right)$,
$\left(V_{\sqrt{2}L}^{T_{\chi_{\mu+\lambda}},\pm},V_{\sqrt{2}L}^{T_{\chi_{\mu},}\pm}\right)$for
any irreducible $\sqrt{2}L/2\sqrt{2}L$-module $T_{\chi_{\mu}}$ such
that $\chi_{\mu}\left(\sqrt{2}\lambda\right)=-1.$

(iv) $M^{1}=V_{\sqrt{2}L}^{T_{\chi_{\mu}},+}$ for an irreducible
$\sqrt{2}L/2\sqrt{2}L$-module $T_{\chi_{\mu}}$, and $\left(M^{2},M^{3}\right)$
is one of the following pairs:

$\left(V_{\frac{\lambda}{\sqrt{2}}+\sqrt{2}L},V_{\sqrt{2}L}^{T_{\chi_{\mu}}^{\left(\frac{\lambda}{\sqrt{2}}\right)},\pm}\right)$
, $\left(V_{\sqrt{2}L}^{T_{\chi_{\mu}}^{\left(\frac{\lambda}{\sqrt{2}}\right)},\pm},V_{\frac{\lambda}{\sqrt{2}}+\sqrt{2}L}\right)$for
$\lambda\in L^{\circ}$ such that $\lambda\notin L$,

$\left(V_{\frac{\lambda}{\sqrt{2}}+\sqrt{2}L}^{\pm},V_{\sqrt{2}L}^{T_{\chi_{\mu+\lambda}},\pm}\right)$
, $\left(V_{\sqrt{2}L}^{T_{\chi_{\mu+\lambda}},\pm},V_{\frac{\lambda}{\sqrt{2}}+\sqrt{2}L}^{\pm}\right)$
for $\lambda\in L$ such that $\chi_{\mu}\left(\sqrt{2}\lambda\right)=1$,

$\left(V_{\frac{\lambda}{\sqrt{2}}+\sqrt{2}L}^{\pm},V_{\sqrt{2}L}^{T_{\chi_{\mu+\lambda}},\mp}\right)$
, $\left(V_{\sqrt{2}L}^{T_{\chi_{\mu+\lambda}},\mp},V_{\frac{\lambda}{\sqrt{2}}+\sqrt{2}L}^{\pm}\right)$for
$\lambda\in L$ such that $\chi_{\mu}\left(\sqrt{2}\lambda\right)=-1$.

(v) $M^{1}=V_{\sqrt{2}L}^{T_{\chi_{\mu}},-}$ for an irreducible $\sqrt{2}L/2\sqrt{2}L$-module
$T_{\chi_{\mu}}$, and $\left(M^{2},M^{3}\right)$ is one of the following
pairs:

$\left(V_{\frac{\lambda}{\sqrt{2}}+\sqrt{2}L},V_{\sqrt{2}L}^{T_{\chi_{\mu}}^{\left(\frac{\lambda}{\sqrt{2}}\right)},\pm}\right)$
, $\left(V_{\sqrt{2}L}^{T_{\chi_{\mu}}^{\left(\frac{\lambda}{\sqrt{2}}\right)},\pm},V_{\frac{\lambda}{\sqrt{2}}+\sqrt{2}L}\right)$
for $\lambda\in L^{\circ}$ such that $\lambda\notin L$,

$\left(V_{\frac{\lambda}{\sqrt{2}}+\sqrt{2}L}^{\pm},V_{\sqrt{2}L}^{T_{\chi_{\mu+\lambda}},\mp}\right)$,
$\left(V_{\sqrt{2}L}^{T_{\chi_{\mu+\lambda}},\mp},V_{\frac{\lambda}{\sqrt{2}}+\sqrt{2}L}^{\pm}\right)$
for $\lambda\in L$ such that $\chi_{\mu}\left(\sqrt{2}\lambda\right)=1$,

$\left(V_{\frac{\lambda}{\sqrt{2}}+\sqrt{2}L}^{\pm},V_{\sqrt{2}L}^{T_{\chi_{\mu+\lambda}},\pm}\right)$,
$\left(V_{\sqrt{2}L}^{T_{\chi_{\mu+\lambda}},\pm},\ V_{\frac{\lambda}{\sqrt{2}}+\sqrt{2}L}^{\pm}\right)$for
$\lambda\in L$ such that $\chi_{\mu}\left(\sqrt{2}\lambda\right)=-1$.

\end{proposition}

\section{The vertex operator algebra $\left(V_{L}\otimes V_{L}\right)^{\mathbb{Z}_{2}}$}

Let $L$ be the positive definite lattice as before. We consider the rational vertex operator algebra $V_L\otimes V_L$
with the natural action of the $2$-cycle $\sigma=(1\ 2).$  We denote the fixed point vertex operator subalgebra
by  $\left(V_{L}\otimes V_{L}\right)^{\mathbb{Z}_{2}}.$

First we want to see $\left(V_{L}\otimes V_{L}\right)^{\mathbb{Z}_{2}}$ is a simple current extension of vertex operator algebra $V_{\sqrt 2 L}\otimes V_{\sqrt 2 L}^+.$ For this purpose,
we let $L^{+}=\left\{ \left(x,x\right)|x\in L\right\} $, $L^{-}=\left\{ \left(x,-x\right)|x\in L\right\} $.
Then
\begin{eqnarray*}
L\oplus L & = & \sum_{\alpha\in L}\left(\left(L^{+}+L^{-}\right)+\left(\alpha,0\right)\right)\\
 & = & \sum_{\alpha\in L}\left(L^{+}+\frac{\left(\alpha,\alpha\right)}{2}\right)\oplus\left(L^{-}+\frac{\left(\alpha,-\alpha\right)}{2}\right)
\end{eqnarray*}

For any $\alpha\in L$, let $\alpha^{1}=\left(\alpha,\alpha\right),\ \alpha^{2}=\left(\alpha,-\alpha\right).$
Then $\left\langle \alpha^{1},\alpha^{1}\right\rangle =\left\langle \alpha^{2},\alpha^{2}\right\rangle =2\left\langle \alpha,\alpha\right\rangle =\left\langle \sqrt{2}\alpha,\sqrt{2}\alpha\right\rangle $.  Note that $\sigma$ also acts on $L\oplus L$ so that
$\sigma\left(\alpha^{1}\right)=\alpha^{1}$ and $\sigma\left(\alpha^{2}\right)=-\alpha^{2}.$
Let $\mathcal S$ be a set of coset representatives of $2L$ in $L.$
Then we have decomposition
\[
V_{L\oplus L}=\sum_{\alpha\in \mathcal{S}}V_{\frac{\alpha^{1}}{2}+L^{+}}\otimes V_{\frac{\alpha^{2}}{2}+L^{-}}.
\]

It is clear that $L^{+}\cong L^{-}\cong\sqrt{2}L.$  Also,  $\sigma(\alpha^2)=\theta(\alpha^2)=-\alpha^2$ where
$\theta$ is the $-1$-isometry on $L^-$ defined before. Thus
\begin{eqnarray*}
\left(V_{L\oplus L}\right)^{\mathbb{Z}_{2}} & \cong & \sum_{\alpha\in \mathcal{S}}V_{\frac{\alpha}{\sqrt{2}}+\sqrt{2}L}\otimes V_{\frac{\alpha}{\sqrt{2}}+\sqrt{2}L.}^{+}
\end{eqnarray*}
For short, we set
\[
\mathcal{U}=\sum_{\alpha\in \mathcal{S}}V_{\frac{\alpha}{\sqrt{2}}+\sqrt{2}L}\otimes V_{\frac{\alpha}{\sqrt{2}}+\sqrt{2}L}^{+}
\]
and
 $$\mathcal{V}=V_{\sqrt{2}L}\otimes V_{\sqrt{2}L}^{+}.$$
It follows from \cite{DL1} and \ref{Fusion-V_L+} that each $V_{\frac{\alpha}{\sqrt{2}}+\sqrt{2}L}\otimes V_{\frac{\alpha}{\sqrt{2}}+\sqrt{2}L}^{+}$ is a simple current $\mathcal V$-module. In particular, $\mathcal{U}$ is a simple current
extension of $\mathcal V$. By \cite{D1, DLM1, A2, DJL}, $\mathcal V$ is rational. We also know that $\mathcal V$ is $C_2$-cofinite \cite{ABD, Ya}.
From  \cite{Y} or \cite{HKL} we have
\begin{proposition}\label{Rationality} The vertex operator algebra
$\mathcal{U}$ is rational.
\end{proposition}

From the classification of irreducible modules of $\mathcal V$ \cite{D1, DN, AD}, every irreducible module has positive weight except the vertex operator algebra itself. A result from \cite{DRX} gives:
\begin{proposition} \label{classificationm} Every irreducible $\mathcal{U}$-module occurs
in an irreducible $\sigma^i$-twisted $V_L\otimes V_L$ module for $i=0, 1.$
\end{proposition}

As far as representation theory concerns, it remains to compute the fusion rules for $\mathcal{U}.$
But it is not so easy to achieve this goal with the irreducible modules given abstractly in \cite{DRX}.
On the other hand   $\mathcal{U}$ is a simple current extension of $V_{\sqrt 2 L}\otimes V_{\sqrt 2 L}^+$
and we know the fusion rules for both $V_{\sqrt 2 L}$ and $V_{\sqrt 2 L}^+,$ it is natural to use these
results to determine the fusion rules for $\mathcal{U}.$ In the rest of this section, we will realize
each irreducible $\mathcal{U}$-module as a direct sum of irreducible $\mathcal V$-modules.

Recall  from \cite{D1} that all irreducible $V_{L}$-modules are given by $V_{L+\lambda}$,
$\lambda\in L^{\circ}$. Let $\mathcal T=\{\lambda_0=0, \lambda_1, \lambda_2, \cdots \}$ be a complete set of representatives of $L$ in $L^\circ$.  Assume that $|\mathcal T|=\left|L^{\circ}/L\right|=$$l$.
Then there are exactly $l$ inequivalent irreducible $V_{L}$-modules.

For any $\lambda,\mu\in L^{\circ}$, $V_{\lambda+L}\otimes V_{\mu+L}$
is an irreducible  $V_{L}\otimes V_{L}$-module. If $\lambda+L\not=\mu+L$, then $V_{\lambda+L}\otimes V_{\mu+L}$
and $V_{\mu+L}\otimes V_{\lambda+L}$ are isomorphic irreducible $\left(V_{L}\otimes V_{L}\right)^{\mathbb{Z}_{2}}$-modules
\cite{DM,DY}. The number of such isomorphism classes of irreducible
$\left(V_{L}\otimes V_{L}\right)^{\mathbb{Z}_{2}}$-modules is $\frac{l^{2}-l}{2}$.

When $\lambda=\mu$, $V_{\lambda+L}\otimes V_{\mu+L}+V_{\mu+L}\otimes V_{\lambda+L}$
split into two different representations of $\left(V_{L}\otimes V_{L}\right)^{\mathbb{Z}_{2}}$
by \cite{DY}. The number of such isomorphism classes of irreducible
$\left(V_{L}\otimes V_{L}\right)^{\mathbb{Z}_{2}}$-modules is $2l$.

It is shown in \cite{BDM} that there is one-to-one correspondence
between the category of $\sigma$-twisted $V_{L}\otimes V_{L}$-modules
and the category of $V_{L}$-modules. Thus the number of isomorphism
classes of irreducible $\sigma$-twisted $V_{L}\otimes V_{L}$-module
is also $l$. From \cite{DY}, each twisted module can be decomposed
into a direct sum of two irreducible $\left(V_{L}\otimes V_{L}\right)^{\mathbb{Z}_{2}}$-modules.
The number of such isomorphism classes of irreducible $\left(V_{L}\otimes V_{L}\right)^{\mathbb{Z}_{2}}$-modules
is $2l$.

Together, we have $\frac{l^{2}+7l}{2}$ inequivalent irreducible $\left(V_{L}\otimes V_{L}\right)^{\mathbb{Z}_{2}}$-modules. The following result is immediate from \cite{DRX}.
\begin{proposition}\label{number of irres} There are exactly  $\frac{l^{2}+7l}{2}$   inequivalent irreducible $\left(V_{L}\otimes V_{L}\right)^{\mathbb{Z}_{2}}$-modules.
\end{proposition}

We now realize each irreducible $\left(V_{L}\otimes V_{L}\right)^{\mathbb{Z}_{2}}$-module in terms of irreducible $\mathcal V$-modules.
\begin{proposition}\label{all modules} Let $L=\mathbb{Z}\alpha_{1}\oplus\cdots\oplus\mathbb{Z}\alpha_{d}$
be a rank $d$ positive definite even lattice. Then any irreducible $\mathcal{U}$-module has the following form
as an $V_{\sqrt 2 L}\otimes V_{\sqrt 2 L}^+$-module:

{\small{}{}(i) For  $\lambda,\mu\in L^{\circ}$ with $\lambda+L\not=\mu+L$, }{\small \par}

\textcolor{blue}{\small{}{}}{\small{}
\[
\left(\lambda\mu\right)=\sum_{\alpha\in \mathcal{S}}V_{\frac{\lambda+\mu+\alpha}{\sqrt{2}}+\sqrt{2}L}\otimes V_{\frac{\lambda-\mu+\alpha}{\sqrt{2}}+\sqrt{2}L.}
\]
}

\textcolor{black}{\small{}{}
(ii) For  $\lambda\in L^{\circ}$,
\[
\widetilde{\left(\lambda\ 0\right)}=\sum_{\alpha\in \mathcal{S}}V_{\sqrt{2}\lambda+\frac{\alpha}{\sqrt{2}}+\sqrt{2}L}\otimes V_{\frac{\alpha}{\sqrt{2}}+\sqrt{2}L}^{+},
\]
}\textcolor{blue}{\small{}{}
\[
{\color{black}\widetilde{\left(\lambda\ 1\right)}=\sum_{\alpha\in \mathcal{S}}V_{\sqrt{2}\lambda+\frac{\alpha}{\sqrt{2}}+\sqrt{2}L}\otimes V_{\frac{\alpha}{\sqrt{2}}+\sqrt{2}L}^{-}}.
\]
}

(iii){\small{}{} For $\lambda\in L^{\circ},$ }{\small \par}

{\small{}{}
\begin{eqnarray*}
\widehat{\left(\lambda\ 0\right)} & = & \sum_{\alpha\in \mathcal{S},\chi_{\lambda}\left(\sqrt{2}\alpha\right)=1}V_{\frac{\lambda+\alpha}{\sqrt{2}}+\sqrt{2}L}\otimes V_{\sqrt{2}L}^{T_{\chi_{\lambda+\alpha}},+}+\sum_{\alpha\in \mathcal{S},\chi_{\lambda}\left({\sqrt{2}\alpha}\right)=-1}V_{\frac{\lambda+\alpha}{\sqrt{2}}+\sqrt{2}L}\otimes V_{\sqrt{2}L}^{T_{\chi_{\lambda+\alpha}},-},
\end{eqnarray*}
}{\small \par}

{\small{}{}
\[
\widehat{\left(\lambda\ 1\right)}=\sum_{\alpha\in \mathcal{S},\chi_{\lambda}\left({\sqrt{2}\alpha}\right)=1}V_{\frac{\lambda+\alpha}{\sqrt{2}}+\sqrt{2}L}\otimes V_{\sqrt{2}L}^{T_{\chi_{\lambda+\alpha}, -}}+\sum_{\alpha\in \mathcal{S}, \chi_{\lambda}\left({\sqrt{2}\alpha}\right)=-1}V_{\frac{\lambda+\alpha}{\sqrt{2}}+\sqrt{2}L}\otimes V_{\sqrt{2}}^{T_{\chi_{\lambda+\alpha}},+}.
\]
}{\small \par}
\end{proposition}

\begin{proof} (i) Let $\lambda,\mu\in L^{\circ},$
with $\lambda+L\not=\mu+L$, then

\begin{align*}
 & \left(\lambda+L\right)\oplus\left(\mu+L\right)\\
 & =\left(\lambda,\mu\right)+\sum_{\alpha\in \mathcal{S}}\left(\frac{\left(\alpha,\alpha\right)}{2}+L^{+}\right)\oplus\left(\frac{\left(\alpha,-\alpha\right)}{2}+L^{-}\right)\\
 & =\sum_{\alpha\in \mathcal{S}}\left(\frac{\lambda+\mu}{\sqrt{2}}+\frac{\alpha}{\sqrt{2}}+\sqrt{2}L\right)\oplus\left(\frac{\lambda-\mu}{\sqrt{2}}+\frac{\alpha}{\sqrt{2}}+\sqrt{2}L\right)
\end{align*}

Notice that here $2\left(\frac{\lambda-\mu}{\sqrt{2}}\right)\notin\sqrt{2}L$.
Thus we obtain the following decomposition
\[
V_{\lambda+L}\otimes V_{\mu+L}\cong\sum_{\alpha\in \mathcal{S}}V_{\frac{\lambda+\mu}{\sqrt{2}}+\frac{\alpha}{\sqrt{2}}+\sqrt{2}L}\otimes V_{\frac{\lambda-\mu}{\sqrt{2}}+\frac{\alpha}{\sqrt{2}}+\sqrt{2}L}
\]
as $V_{\sqrt 2 L}\otimes V_{\sqrt 2 L}^+$-modules. The number of such pairs $\left(\lambda, \ \mu\right)$ that give inequivalent irreducible  $\mathcal{U}$-modules
is $\frac{l^{2}-l}{2}$. We obtain  $\frac{l^{2}-l}{2}$
irreducible $\mathcal{U}$-modules in this way. Denote these modules by $\left(\lambda\ \mu\right), \lambda, \mu\in \mathcal T.$

(ii) If $\lambda=\mu$, then $V_{\frac{\lambda-\mu}{\sqrt{2}}+\frac{\alpha}{\sqrt{2}}+\sqrt{2}L}\cong V_{\frac{\alpha}{\sqrt{2}}+\sqrt{2}L}^{+}+V_{\frac{\alpha}{\sqrt{2}}+\sqrt{2}L}^{-}$
as $V_{\sqrt{2}L}^{+}$-modules. Thus we have
\begin{align*}
V_{\lambda+L}\otimes V_{\lambda+L} & =\left(\sum_{\alpha\in \mathcal{S}}V_{\sqrt{2}\lambda+\frac{\alpha}{\sqrt{2}}+\sqrt{2}L}\otimes V_{\frac{\alpha}{\sqrt{2}}+\sqrt{2}L}^{+}\right)\\
 & +\left(\sum_{\alpha\in \mathcal{S}}V_{\sqrt{2}\lambda+\frac{\alpha}{\sqrt{2}}+\sqrt{2}L}\otimes V_{\frac{\alpha}{\sqrt{2}}+\sqrt{2}L}^{-}\right)
\end{align*}
which is a sum of two irreducible $\mathcal{U}$-modules. Denote $\sum_{\alpha\in \mathcal{S}}V_{\sqrt{2}\lambda+\frac{\alpha}{\sqrt{2}}+\sqrt{2}L}\otimes V_{\frac{\alpha}{\sqrt{2}}+\sqrt{2}L}^{+}$
by $\widetilde{\left(\lambda\ 0\right)}$ and $\sum_{\alpha\in \mathcal{S}}V_{\sqrt{2}\lambda+\frac{\alpha}{\sqrt{2}}+\sqrt{2}L}\otimes V_{\frac{\alpha}{\sqrt{2}}+\sqrt{2}L}^{-}$
by $\widetilde{\left(\lambda\ 1\right)}$. Then $\widetilde{\left(\lambda\ \epsilon\right)},$
$\lambda\in \mathcal T, \epsilon=0, 1$ give all inequivalent irreducible $\mathcal U$-modules of this form. The number of such inequivalent irreducible $\mathcal{U}$-modules  is $2l.$

(iii) The proof in this case is different from the cases (i), (ii). It is
difficult to identify the irreducible $\mathcal U$-modules  from decomposition of
$\sigma$-twisted modules of $V_{L}\otimes V_{L}$ directly. Note that
$V_{\sqrt{2}L}\otimes V_{\sqrt{2}L}$ is a subalgebra of $V_{L}\otimes V_{L}$
(see the discussion before Proposition \ref{Rationality}). So any irreducible $\sigma$-twisted
$V_{L}\otimes V_{L}$-module contains an irreducible $1\otimes \theta$-twisted $V_{\sqrt{2}L}\otimes V_{\sqrt{2}L}$-module $V_{\frac{\lambda}{\sqrt{2}}+\sqrt{2}L}\otimes V_{\sqrt{2}L}^{T_{\chi_{\mu}}}$ for some $\lambda, \mu \in L^\circ$ \cite{D2}. Moreover, $V_{\frac{\lambda}{\sqrt{2}}+\sqrt{2}L}\otimes V_{\sqrt{2}L}^{T_{\chi_{\mu}}}$
is a direct sum of two irreducible  $\mathcal{V}$-modules $V_{\frac{\lambda}{\sqrt{2}}+\sqrt{2}L}\otimes V_{\sqrt{2}L}^{T_{\chi_{\mu},\pm}}.$

 Since $\mathcal{V}$ is a rational vertex operator subalgebra of $\mathcal{U}$, each irreducible
$\mathcal{U}$-module $M$ is a direct sum of irreducible $\mathcal{V}$-modules.
For an irreducible $\mathcal{V}$-module $W$, we define $\mathcal{U}\cdot W$
to be the fusion product of $\mathcal{U}$ and $W$ as $\mathcal{V}$-modules.
That is,
\[
\mathcal{U}\cdot W=\left(\sum_{\alpha\in \mathcal{S}}V_{\frac{\alpha}{\sqrt{2}}+\sqrt{2}L}\otimes V_{\frac{\alpha}{\sqrt{2}}+\sqrt{2}L}^{+}\right)\boxtimes_{\mathcal{V}}W.
\]
Since each $V_{\frac{\alpha}{\sqrt{2}}+\sqrt{2}L}\otimes V_{\frac{\alpha}{\sqrt{2}}+\sqrt{2}L}^{+}$ is a simple current, $\mathcal{U}\cdot W$ is a $\mathcal{U}$-module if and only if the weights of $\mathcal{U}\cdot W$
lies in $\mathbb{Z}+r$ for some $r\in {\mathbb C}.$

From the discussion above, we see that any irreducible $\mathcal{U}$-modules from the $\sigma$-twisted $V_{L}\otimes V_{L}$-modules has the form $\mathcal U\cdot W$ where $W=V_{\frac{\lambda}{\sqrt{2}}+\sqrt{2}L}\otimes V_{\sqrt{2}L}^{T_{\chi_{\mu},\pm}}$ and $\lambda, \mu \in L^\circ.$ First we take
$W=V_{\frac{\lambda}{\sqrt{2}}+\sqrt{2}L}\otimes V_{\sqrt{2}L}^{T_{\chi_{\mu},+}}.$
By fusion rules in Proposition
\ref{Fusion-V_L} and Proposition \ref{Fusion-V_L+}, we get
\begin{align*}
\mathcal{U}\cdot W & =\sum_{\alpha\in \mathcal{S},\chi_{\mu}\left({\sqrt{2}\alpha}\right)=1}V_{\frac{\lambda+\alpha}{\sqrt{2}}+\sqrt{2}L}\otimes V_{\sqrt{2}}^{T_{\chi_{\mu+\alpha}},+}+\sum_{\alpha\in \mathcal{S},\chi_{\mu}\left({\sqrt{2}\alpha}\right)=-1}V_{\frac{\lambda+\alpha}{\sqrt{2}}+\sqrt{2}L}\otimes V_{\sqrt{2}}^{T_{\chi_{\mu+\alpha}},-}.
\end{align*}

Thus $\mathcal{U}\cdot W$ is a $\mathcal{U}$-module only if $\left\langle \lambda,\alpha\right\rangle +\frac{\left\langle \alpha,\alpha\right\rangle }{2}\in 2\mathbb{Z}$ for each
$\alpha\in S$ that satisfies $\chi_{\mu}\left({\sqrt{2}\alpha}\right)=1$, and $\left\langle \lambda,\alpha\right\rangle +\frac{\left\langle \alpha,\alpha\right\rangle }{2}\in 2\mathbb{Z}+1$ for each $\alpha\in S$ that satisfies $\chi_{\mu}\left({\sqrt{2}\alpha}\right)=-1$. So $\lambda,\mu$ must satisfy $\left\langle \lambda-\mu,\alpha\right\rangle \in2\mathbb{Z}$.
That is, $\chi_{\lambda}\left(\sqrt{2}\alpha_{i}\right)=\chi_{\mu}\left(\sqrt{2}\alpha_{i}\right)$
for any $1\le i\le d$ and hence $\chi_\lambda$ and $\chi_{\mu}$
define the same character.  Thus
\[
\mathcal{U}\cdot W=\sum_{\alpha\in \mathcal{S},\chi_{\lambda}\left({\sqrt{2}\alpha}\right)=1}V_{\frac{\lambda+\alpha}{\sqrt{2}}+\sqrt{2}L}\otimes V_{\sqrt{2}L}^{T_{\chi_{\lambda+\alpha}},+}+\sum_{\alpha\in \mathcal{S},\chi_{\mu}\left({\sqrt{2}\alpha}\right)=-1}V_{\frac{\lambda+\alpha}{\sqrt{2}}+\sqrt{2}L}\otimes V_{\sqrt{2}L}^{T_{\chi_{\lambda+\alpha}},-}
\]
 is an irreducible $\mathcal{U}$-module.
We denote this module by $\widehat{\left(\lambda\ 0\right)}$.

We now prove that $\widehat{\left(\lambda+\beta\ 0\right)}=\widehat{\left(\lambda\ 0\right)}$ for any $\beta\in L.$ It is clear that for any $\beta\in 2L,$
\[
V_{\frac{\lambda+\beta+\alpha}{\sqrt{2}}+\sqrt{2}L}\otimes V_{\sqrt{2}L}^{T_{\chi_{\lambda+\beta+\alpha}},+}
=V_{\frac{\lambda+\alpha}{\sqrt{2}}+\sqrt{2}L}\otimes V_{\sqrt{2}L}^{T_{\chi_{\lambda+\alpha}},+}
\]
for $\alpha\in {\cal S}.$ So $\widehat{\left(\lambda+\beta\ 0\right)}=\widehat{\left(\lambda\ 0\right)}.$
If $\beta$ does not lie in $2L,$ we can assume $\beta\in \mathcal S$ as $\mathcal S$ is a coset representatives of $2L$ in $L.$
The result follows immediately from the definition of $\mathcal{U}\cdot W.$

Similarly, we can prove that for any $\lambda,\mu\in L^{\circ},$
when $W=V_{\frac{\lambda}{\sqrt{2}}+\sqrt{2}L}\otimes V_{\sqrt{2}L}^{T_{\chi_{\mu},-}}$,
\[
\mathcal{U}\cdot W=\sum_{\alpha\in \mathcal{S},\chi_{\lambda}\left({\sqrt{2}\alpha}\right)=1}V_{\frac{\lambda+\alpha}{\sqrt{2}}+\sqrt{2}L}\otimes V_{\sqrt{2}L}^{T_{\chi_{\lambda+\alpha}},-}+\sum_{\alpha\in \mathcal{S},\chi_{\mu}\left({\sqrt{2}\alpha}\right)=-1}V_{\frac{\lambda+\alpha}{\sqrt{2}}+\sqrt{2}L}\otimes V_{\sqrt{2}L}^{T_{\chi_{\lambda+\alpha}},+}
\]
is an irreducible $\mathcal{U}$-module which we denote by $\widehat{\left(\lambda\ 1\right)}.$

The number of inequivalent  irreducible $\mathcal{U}$-modules of the form $\widehat{\left(\lambda\ \epsilon\right)}$,
$\lambda\in \mathcal T$, $\epsilon=0,1$ is $2l$.

Now we have in total $\frac{l^{2}-l}{2}+2l+2l=\frac{l^{2}+7l}{2}$
irreducible $\mathcal{U}$-modules. Thus they are the inequivalent
irreducible $\mathcal{U}$-modules. \end{proof}

\begin{remark} \label{Dual of U-modules} By Proposition 3.7 in \cite{ADL},
for any irreducible $V_{\sqrt{2}L}$-module $V_{\lambda_{i}+\sqrt{2}L}$,
we have $\left(V_{\lambda_{i}+\sqrt{2}L}\right)^{'}=V_{-\lambda_{i}+\sqrt{2}L}$,
where $\lambda_{i}$ is any coset representative of $\sqrt{2}L$ in
$\left(\sqrt{2}L\right)^{\circ}$. By Proposition 3.3, it is clear
that for any $\lambda,\mu\in L^{\circ}$ with $\lambda+L\not=\mu+L$,
and $\epsilon=0,1$, we have $\left(\lambda\ \mu\right)^{'}=\left(-\lambda\ -\mu\right)$,
$\widetilde{\left(\lambda\ \epsilon\right)}^{'}=\widetilde{\left(-\lambda\ \epsilon\right)}$,
and $\widehat{\left(\lambda\ \epsilon\right)}^{'}=\widehat{\left(-\lambda\ \epsilon\right)}$.
\end{remark}

\section{The quantum dimensions }
\setcounter{equation}{0}
Quantum dimensions have been systematically studied in \cite{DXY, DRX}.
It is proved that for a rational, $C_{2}$-cofinite, self-dual vertex
operator algebra of CFT type, quantum dimensions of its irreducible
modules have nice properties that are helpful in determining fusion
products. The 2-permuation orbifold model $\left(V_{L}\otimes V_{L}\right)^{\mathbb{Z}_{2}}$
we study here satisfies all the conditions and hence we can use quantum
dimensions to determine some fusion rules. First we recall some notions
and properties about quantum dimensions.

\begin{definition} Let $g$ be an automorphism of the vertex operator
algebra $V$ with order $T$. Let $M=\oplus_{n\in\frac{1}{T}\mathbb{Z}_{+}}M_{\lambda+n}$
be a $g$-twisted $V$-module, the formal character of $M$ is defined
as

\[
\mbox{ch}_{q}M=\mbox{tr}_{M}q^{L\left(0\right)-c/24}=q^{\lambda-c/24}\sum_{n\in\frac{1}{T}\mathbb{Z}_{+}}\left(\dim M_{\lambda+n}\right)q^{n},
\]
where $\lambda$ is the conformal weight of $M$. \end{definition}

We denote the holomorphic function $\mbox{ch}_{q}M$ by $Z_{M}\left(\tau\right)$.
Here and below, $\tau$ is in the upper half plane $\mathbb{H}$ and
$q=e^{2\pi i\tau}$.

\begin{definition} \label{quantum dimension0}Let $V$ be a vertex
operator algebra and $M$ a $g$-twisted $V$-module such that $Z_{V}\left(\tau\right)$
and $Z_{M}\left(\tau\right)$ exists. The quantum dimension of $M$
over $V$ is defined as
\[
\mbox{qdim}_{V}M=\lim_{y\to0}\frac{Z_{M}\left(iy\right)}{Z_{V}\left(iy\right)},
\]
where $y$ is real and positive. \end{definition}

Assume $V$ is a rational, $C_{2}$-cofinite vertex operator algebra
of CFT type with $V\cong V'$. Let $M^{0}\cong V,\,M^{1},\,\cdots,\,M^{l}$
be all inequivalent irreducible $V$-modules. Moreover, we assume
the conformal weights $\lambda_{i}$ of $M^{i}$ are positive for
all $i>0.$ Then we have the following properties of quantum dimensions
\cite{DJX}:

\begin{proposition}\label{possible values of quantum dimensions}
$q\dim_{V}M^{i}\geq1,$ $\forall i=0,\cdots,l.$

\end{proposition}

\begin{proposition}\label{quantum-product} For any $i,\,j=0,\cdots,\,l,$
\[
q\dim_{V}\left(M^{i}\boxtimes M^{j}\right)=q\dim_{V}M^{i}\cdot q\dim_{V}M^{j}.
\]
\end{proposition}

\begin{proposition}\label{simple current quantum dim}A $V$-module
$M$ is a simple current if and only if $q\dim_{V}M=1$.\end{proposition}

\begin{remark} By Proposition \ref{Rationality} and \cite{A4} we
see that the vertex operator algebra $\left(V_{L}\otimes V_{L}\right)^{\mathbb{Z}_{2}}$
satisfies all the assumptions required in \cite{DJX} and thus we can apply these properties.
\end{remark}

We obtain quantum dimensions of all irreducible $\left(V_{L}\otimes V_{L}\right)^{\mathbb{Z}_{2}}$-modules
as follows:

\begin{proposition}\label{quantum dimension} For $\lambda,\mu\in L^{\circ}$ with $\lambda+L\not=\mu+L$,
$\epsilon=0,1$, we have

\begin{equation}
q\dim_{\mathcal{U}}\widetilde{\left(\lambda\ \epsilon\right)}=1\label{qum-dim-diag}
\end{equation}

\begin{equation}
q\dim_{\mathcal{U}}\left(\lambda\ \mu\right)=2\label{qum-dim-nondiag}
\end{equation}

\begin{equation}
q\dim_{\mathcal{U}}\widehat{\left(\lambda\ \epsilon\right)}=\sqrt{\left|L^{\circ}/L\right|}\ \label{qum-dim-twisted}
\end{equation}

\end{proposition}

\begin{proof} Using the definition of quantum dimension we see  that for any  irreducible $\mathcal{U}$-module $M$,
$$q\dim_{\mathcal{U}}M=\frac{q\dim_{\mathcal{V}}M}{q\dim_{\mathcal{V}}\mathcal{U}}.$$

For any $\alpha \in L^\circ, \alpha \in \mathcal S$, by  fusion rules of irreducible $V_{\sqrt 2 L}$- and $V_{\sqrt 2 L}^+$-modules
in Proposition \ref{Fusion-V_L} and Proposition \ref{Fusion-V_L+},
we see that $V_{\sqrt{2}\lambda+\frac{\alpha}{\sqrt{2}}+\sqrt{2}L}\otimes V_{\frac{\alpha}{\sqrt{2}}+\sqrt{2}L}^{+}$
is a simple current $\mathcal V$-module. By   Proposition \ref{quantum-product} and Proposition \ref{simple current quantum dim} ,
$V_{\sqrt{2}\lambda+\frac{\alpha}{\sqrt{2}}+\sqrt{2}L}\otimes V_{\frac{\alpha}{\sqrt{2}}+\sqrt{2}L}^{+}$
is of quantum dimension 1 as irreducible $\mathcal{V}$-module.
Thus $q\dim_{\mathcal{V}}\widetilde{\left(\lambda\ \epsilon\right)}=2^{d}$ and $q\dim_{\mathcal{V}}\mathcal{U}=2^{d}.$
Thus we obtain $q\dim_{\mathcal{U}}\widetilde{\left(\lambda\ \epsilon\right)}=1$,
$\epsilon=0,1$.

For $\alpha \in \mathcal S$, $\lambda,\mu\in L^{\circ}$ with $\lambda+L\not=\mu+L$,  we have $2\left(\frac{\lambda-\mu+\alpha}{\sqrt{2}}\right)\not\in\sqrt{2}L$. Thus by fusion rules of irreducible $V_{\sqrt{2}L}^{+}$-modules in
Proposition  \ref{Fusion-V_L+}, quantum dimension of $V_{\frac{\lambda-\mu+\alpha}{\sqrt{2}}+\sqrt{2}L}$
is $2$ as irreducible $V_{\sqrt{2}L}^{+}$-module. By Proposition \ref{quantum-product}, $$q\dim_{\mathcal{V}}(V_{\frac{\lambda+\mu+\alpha}{\sqrt{2}}+\sqrt{2}L}\otimes V_{\frac{\lambda-\mu+\alpha}{\sqrt{2}}+\sqrt{2}L})=q\dim_{\sqrt{2}L}V_{\frac{\lambda+\mu+\alpha}{\sqrt{2}}+\sqrt{2}L}\cdot q\dim_{V_{\sqrt{2}L}^{+}}V_{\frac{\lambda+\mu+\alpha}{\sqrt{2}}+\sqrt{2}L}=2.$$

Therefore we get $q\dim_{\mathcal{V}}\left(\lambda\ \mu\right)=2\cdot2^{d}=2^{d+1}$ and hence we prove (\ref{qum-dim-nondiag}).

To prove (\ref{qum-dim-twisted}), first we recall from \cite{DJX} that $\mbox{glob}\left(V\right)=\sum_{M}\left(q\dim M\right)^{2}$
where $M$ runs over all irreducible modules of $V$. By Proposition \ref{quantum-product}, $\mbox{glob}\left(V_{L}\otimes V_{L}\right)=\left(\mbox{glob}\left(V_{L}\right)\right)^{2}$. Now we have
\[
\mbox{glob}\left(V_{L}\otimes V_{L}\right)=\left(\sum_{\lambda\in \mathcal T}\left(q\dim V_{\lambda+L}\right)^{2}\right)^{2}=|\mathcal T|^2=\left|L^{\circ}/L\right|^{2}=l^{2}
\]
as $q\dim V_{\lambda+L}=1$ for any irreducible $V_{L}$-module $V_{\lambda+L}$. Set $q\dim_{\mathcal{U}}\widehat{\left(\lambda\ \epsilon\right)}=x$
, by quantum dimensions of irreducible $\mathcal{U}$-modules $\widetilde{\left(\lambda\ \epsilon\right)}$
and $\left(\lambda\ \mu\right)$ we obtain above, we have
\begin{align*}
\mbox{glob}\left(V^{G}\right) & =\frac{l^{2}-l}{2}\cdot2^{2}+2l\cdot1^{2}+2l\cdot x^{2}.
\end{align*}

It is proved in \cite{DRX} that $\mbox{glob}\left(V^{G}\right)=\left|G\right|^{2}\mbox{glob}\left(V\right).$
Therefore we get $2^{2}\cdot\frac{l^{2}-l}{2}+2l\cdot1^{2}+2l\cdot x^{2}=2^{2}\cdot l^{2}$.
Solving the equation gives $x=\sqrt l$ and thus
$q\dim_{\mathcal{U}}\widehat{\left(\lambda\ \epsilon\right)}=\sqrt{\left|L^{\circ}/L\right|}$.

\end{proof}

\section{Fusion Rules }
\setcounter{equation}{0}
In this section, we use the quantum dimensions obtained in the previous
section and the fusion rules of irreducible $V_{\sqrt{2}L}$- and
$V_{\sqrt{2}L}^{+}$-modules in \cite{DL1,ADL} to determine the fusion
products of the 2-permutation orbifold model.

\begin{theorem}
%%%%%%%% fusion product %%%%%%%%%%%%%%%%%%

Let $L$ be as before. Let $\lambda,\mu,\gamma,\delta\in L^{\circ},$
$\epsilon,\epsilon_{1}=0,1.$

(a) (i)

\begin{equation}
\widetilde{\left(\lambda\ \epsilon\right)}\boxtimes\widetilde{\left(\gamma\ \epsilon_{1}\right)}=\widetilde{\left(\lambda+\gamma\ \epsilon+\epsilon_{1}\right)},\label{Fusion-Diag-Diag}
\end{equation}

(ii) if $\gamma+L\not=\delta+L$, then
\begin{equation}
\widetilde{\left(\lambda\ \epsilon\right)}\boxtimes\left(\gamma\ \delta\right)=\left(\lambda+\gamma\ \lambda+\delta\right),\label{Fusion-Nondiag-Diag}
\end{equation}

(iii) if $\lambda+L\not=\mu+L$ , $\gamma+L\not=\delta+L$, $\lambda+\gamma+L=\mu+\delta+L$,
and $\mu+\gamma+L=\lambda+\delta+L$, then
\begin{equation}
\left(\lambda\ \mu\right)\boxtimes\left(\gamma\ \delta\right)=\widetilde{\left(\lambda+\gamma\ 0\right)}+\widetilde{\left(\lambda+\gamma\ 1\right)}+\widetilde{\left(\mu+\gamma\ 0\right)}+\widetilde{\left(\mu+\gamma\ 1\right)},\label{Fusion-Nondiag-Nodiag}
\end{equation}

(iv) if $\lambda+L\not=\mu+L$ , $\gamma+L\not=\delta+L$, $\lambda+\gamma+L\not=\mu+\delta+L$,
and $\mu+\gamma+L=\lambda+\delta+L$, then
\begin{equation}
\left(\lambda\ \mu\right)\boxtimes\left(\gamma\ \delta\right)=\left(\lambda+\gamma\ \mu+\delta\right)+\widetilde{\left(\mu+\gamma\ 0\right)}+\widetilde{\left(\mu+\gamma\ 1\right)},\label{add}
\end{equation}

(v) if $\lambda+L\not=\mu+L$, $\gamma+L\not=\delta+L$, $\lambda+\gamma+L\not=\mu+\delta+L$
and $\mu+\gamma+L\not=\lambda+\delta+L$, then

\begin{equation}
\left(\lambda\ \mu\right)\boxtimes\left(\gamma\ \delta\right)=\left(\lambda+\gamma\ \mu+\delta\right)+\left(\mu+\gamma\ \lambda+\delta\right).\label{Fusion-Nondiag-Nondiag-Case 2}
\end{equation}

(b) If $\lambda+L\not=\mu+L$,

\begin{equation}
\left(\lambda\ \mu\right)\boxtimes\widehat{\left(\gamma\ \epsilon\right)}=\widehat{\left(\lambda+\mu+\gamma\ 0\right)}+\widehat{\left(\lambda+\mu+\gamma\ 1\right)},\epsilon=0,1,\label{Fusion-Nondiag-Twisted}
\end{equation}

\begin{equation}
\widetilde{\left(\lambda\ \epsilon\right)}\boxtimes\widehat{\left(\mu\ \epsilon_{1}\right)}=\widehat{\left(2\lambda+\mu\ \epsilon+\epsilon_{1}\right)}.\label{Fusion Prod-dia-twisted}
\end{equation}

(c) Let $G=\left\{ \gamma\in{\cal T}|2\gamma\in L\right\} $,

(i) if $\frac{\lambda+\mu}{2}\in L^{\circ},$
\begin{equation}
\widehat{\left(\lambda\ \epsilon\right)}\boxtimes\widehat{\left(\mu\ \epsilon_{1}\right)}=\sum_{\gamma\in G}\widetilde{\left(\frac{\lambda+\mu}{2}+\gamma\ \epsilon-\epsilon_{1}\right)}+\sum_{\delta\in L^{\circ},\delta\not=\frac{\mu+\lambda}{2}+\gamma,\gamma\in G}\left(\lambda+\mu-\delta\ \delta\right),\label{Fusion-Twisted-Twisted-Case 1}
\end{equation}

(ii) if $\frac{\lambda+\mu}{2}\not\in L^{\circ}$,
\begin{equation}
\widehat{\left(\lambda\ \epsilon\right)}\boxtimes\widehat{\left(\mu\ \epsilon_{1}\right)}=\sum_{\delta\in L^{\circ},\delta\not\not=\lambda+\mu-\delta}\left(\lambda+\mu-\delta\ \delta\right).\label{Fusion-Twisted-Twisted-Case 2}
\end{equation}

\end{theorem}

\begin{proof} Consider fusion product of $V_{\lambda+L}\otimes V_{\mu+L}$
and $V_{\gamma+L}\otimes V_{\delta+L}$ as irreducible $V_{L}\otimes V_{L}$-modules.
By fusion rules in Proposition \ref{Fusion-V_L} and Theorem 2.10
in \cite{ADL}, as irreducible $V_{L}\otimes V_{L}$-modules, we have
the following fusion product:

\begin{equation}
\left(V_{\lambda+L}\otimes V_{\mu+L}\right)\boxtimes\left(V_{\gamma+L}\otimes V_{\delta+L}\right)=V_{\lambda+\gamma+L}\otimes V_{\mu+\delta+L}.\label{V_L-fusion-1}
\end{equation}

\emph{Proof of (\ref{Fusion-Diag-Diag}):} If $\lambda+L=\mu+L$ and
$\gamma+L=\delta+L$, then $\lambda+\gamma+L=\mu+\delta+L$. Moreover,
$V_{\lambda+L}\otimes V_{\mu+L}\cong\widetilde{\left(\lambda\ 0\right)}+\widetilde{\left(\lambda\ 1\right)}$,
$V_{\gamma+L}\otimes V_{\delta+L}\cong\widetilde{\left(\gamma\ 0\right)}+\widetilde{\left(\gamma\ 1\right)},$
and $V_{\lambda+\gamma+L}\otimes V_{\mu+\delta+L}\cong\widetilde{\left(\lambda+\gamma\ 0\right)}+\widetilde{\left(\lambda+\gamma\ 1\right)}$
as $\mathcal{U}$-modules.

First we consider fusion rule $N_{\mathcal{U}}\left(_{\widetilde{\left(\lambda\ 0\right)}\ \widetilde{\left(\gamma\ 0\right)}}^{\widetilde{\left(\lambda+\gamma\ \epsilon '\right)}}\right)$,
$\epsilon '\in\left\{ 0,1\right\} $. Take $V=V_{L}\otimes V_{L}$ and
$U=\mathcal{U}$ in Proposition 2.9 in \cite{ADL}, then (\ref{V_L-fusion-1})
implies
\[
1=N_{V_{L}\otimes V_{L}}\left(_{V_{\lambda+L}\otimes V_{\mu+L}\ V_{\gamma+L}\otimes V_{\delta+L}}^{V_{\lambda+\gamma+L}\otimes V_{\mu+\delta+L}}\right)\le N_{\mathcal{U}}\left(_{\widetilde{\left(\lambda\ 0\right)}\ \widetilde{\left(\gamma\ 0\right)}}^{\widetilde{\left(\lambda+\gamma\ 0\right)}+\widetilde{\left(\lambda+\gamma\ 1\right)}}\right).
\]
So $N_{\mathcal{U}}\left(_{\widetilde{\left(\lambda\ 0\right)}\ \widetilde{\left(\gamma\ 0\right)}}^{\widetilde{\left(\lambda+\gamma\ \epsilon'\right)}}\right)=0$
or 1.

Now take $V=\mathcal{U}$ and $U=\mathcal{V}$ in Proposition 2.9
in \cite{ADL}, then
\[
N_{\mathcal{U}}\left(_{\widetilde{\left(\lambda\ 0\right)}\ \widetilde{\left(\gamma\ 0\right)}}^{\widetilde{\left(\lambda+\gamma\ \epsilon'\right)}}\right)\le N_{\mathcal{V}}\left(_{V_{\sqrt{2}\lambda+\sqrt{2}L}\otimes V_{\sqrt{2}L}^{+}\ V_{\sqrt{2}\gamma+\sqrt{2}L}\otimes V_{\sqrt{2}L}^{+}}^{\widetilde{\left(\lambda+\gamma\ \epsilon'\right)}}\right)
\]
Since $\widetilde{\left(\lambda+\gamma\ \epsilon'\right)}=\sum_{\alpha\in\mathcal{S}}V_{\sqrt{2}\left(\lambda+\gamma\right)+\frac{\alpha}{\sqrt{2}}+\sqrt{2}L}\otimes V_{\frac{\alpha}{\sqrt{2}}+\sqrt{2}L}^{\pm}$
where $\pm$ depends on value of $\epsilon',$ we have
\begin{align*}
 & N_{\mathcal{V}}\left(_{V_{\sqrt{2}\lambda+\sqrt{2}L}\otimes V_{\sqrt{2}L}^{+}\ V_{\sqrt{2}\gamma+\sqrt{2}L}\otimes V_{\sqrt{2}L}^{+}}^{\widetilde{\left(\lambda+\gamma\ \epsilon'\right)}}\right)\\
 & =\sum_{\alpha\in\mathcal{S}}N_{\mathcal{V}}\left(_{V_{\sqrt{2}\lambda+\sqrt{2}L}\otimes V_{\sqrt{2}L}^{+}\ V_{\sqrt{2}\gamma+\sqrt{2}L}\otimes V_{\sqrt{2}L}^{+}}^{V_{\sqrt{2}\left(\lambda+\gamma\right)+\frac{\alpha}{\sqrt{2}}+\sqrt{2}L}\otimes V_{\frac{\alpha}{\sqrt{2}}+\sqrt{2}L}^{\pm}}\right)\\
 & =\sum_{\alpha\in\mathcal{S}}N_{V{}_{\sqrt{2}L}}\left(_{V_{\sqrt{2}\lambda+\sqrt{2}L}\ V_{\sqrt{2}\gamma+\sqrt{2}L}}^{V_{\sqrt{2}\left(\lambda+\gamma\right)+\frac{\alpha}{\sqrt{2}}+\sqrt{2}L}}\right)\cdot N_{V_{\sqrt{2}L}^{+}}\left(_{V_{\sqrt{2}L}^{+}\ V_{\sqrt{2}L}^{+}}^{V_{\frac{\alpha}{\sqrt{2}}+\sqrt{2}L}^{\pm}}\right).
\end{align*}

It is clear that
\[
N_{V{}_{\sqrt{2}L}}\left(_{V_{\sqrt{2}\lambda+\sqrt{2}L}\ V_{\sqrt{2}\gamma+\sqrt{2}L}}^{V_{\sqrt{2}\left(\lambda+\gamma\right)+\frac{\alpha}{\sqrt{2}}+\sqrt{2}L}}\right)=N_{V_{\sqrt{2}L}^{+}}\left(_{V_{\sqrt{2}L}^{+}\ V_{\sqrt{2}L}^{+}}^{V_{\frac{\alpha}{\sqrt{2}}+\sqrt{2}L}^{+}}\right)=1
\]
if and only if $\alpha=0$. $N_{V_{\sqrt{2}L}^{+}}\left(_{V_{\sqrt{2}L}^{+}\ V_{\sqrt{2}L}^{+}}^{V_{\frac{\alpha}{\sqrt{2}}+\sqrt{2}L}^{-}}\right)=0$
for $\alpha\in\mathcal{S}$ forces $\epsilon'=0$. So $N_{\mathcal{U}}\left(_{\widetilde{\left(\lambda\ 0\right)}\ \widetilde{\left(\gamma\ 0\right)}}^{\widetilde{\left(\lambda+\gamma\ 0\right)}}\right)\le1$
and $N_{\mathcal{U}}\left(_{\widetilde{\left(\lambda\ 0\right)}\ \widetilde{\left(\gamma\ 0\right)}}^{\widetilde{\left(\lambda+\gamma\ 1\right)}}\right)=0$.
By quantum dimensions in Proposition \ref{quantum dimension}, we
get $\widetilde{\left(\lambda\ 0\right)}\boxtimes\widetilde{\left(\gamma\ 0\right)}=\widetilde{\left(\lambda+\gamma\ 0\right)}$.
We can similarly prove that for $\epsilon,\epsilon_{1}=0,1$, we have
\[
\widetilde{\left(\lambda\ \epsilon\right)}\boxtimes\widetilde{\left(\gamma\ \epsilon_{1}\right)}=\widetilde{\left(\lambda+\gamma\ \epsilon+\epsilon_{1}\right)}.
\]
Thus (\ref{Fusion-Diag-Diag}) has been proved.

\emph{Proof of (\ref{Fusion-Nondiag-Diag}):} We now have $V_{\lambda+L}\otimes V_{\lambda+L}\cong\widetilde{\left(\lambda\ 0\right)}+\widetilde{\left(\lambda\ 1\right)},$
$V_{\gamma+L}\otimes V_{\delta+L}\cong\left(\gamma\ \delta\right)$,
and $V_{\lambda+\gamma+L}\otimes V_{\lambda+\delta+L}\cong\left(\lambda+\gamma\ \lambda+\delta\right)$
as irreducible $\mathcal{U}$-modules. Take $V=V_{L}\otimes V_{L}$
and $U=\mathcal{U}$ in Proposition 2.9 in \cite{ADL}, then (\ref{V_L-fusion-1})
implies
\[
1=N_{V_{L}\otimes V_{L}}\left(_{V_{\lambda+L}\otimes V_{\lambda+L}\ V_{\gamma+L}\otimes V_{\delta+L}}^{V_{\lambda+\gamma+L}\otimes V_{\lambda+\delta+L}}\right)\le N_{\mathcal{U}}\left(_{\widetilde{\left(\lambda\ \epsilon\right)}\ \left(\gamma\ \delta\right)}^{\left(\lambda+\gamma\ \lambda+\delta\right)}\right)
\]
By quantum dimensions in Proposition \ref{quantum dimension} and
Proposition \ref{quantum-product}, we see that
\[
q\dim_{\mathcal{U}}\left(\widetilde{\left(\lambda\ \epsilon\right)}\boxtimes\left(\gamma\ \delta\right)\right)=2.
\]

So $N_{\mathcal{U}}\left(_{\widetilde{\left(\lambda\ \epsilon\right)}\ \left(\gamma\ \delta\right)}^{\left(\lambda+\gamma\ \lambda+\delta\right)}\right)=1$
and hence $\widetilde{\left(\lambda\ \epsilon\right)}\boxtimes\left(\gamma\ \delta\right)=\left(\lambda+\gamma\ \mu+\delta\right),$
as desired.

\emph{Proof of (\ref{Fusion-Nondiag-Nodiag}):} We have $V_{\lambda+L}\otimes V_{\mu+L}\cong\left(\lambda\ \mu\right)$,
$V_{\gamma+L}\otimes V_{\delta+L}\cong\left(\gamma\ \delta\right)$,
$V_{\lambda+\gamma+L}\otimes V_{\mu+\delta+L}\cong\widetilde{\left(\lambda+\gamma\ 0\right)}+\widetilde{\left(\lambda+\gamma\ 1\right)}$,
and $V_{\mu+\gamma+L}\otimes V_{\lambda+\delta+L}\cong\widetilde{\left(\mu+\gamma\ 0\right)}+\widetilde{\left(\mu+\gamma\ 1\right)}$
as $\mathcal{U}$-modules. By the fusion product in (\ref{V_L-fusion-1}),
we know there is a nonzero intertwining operator
\[
I\in I_{V_{L}\otimes V_{L}}\left(\begin{array}{c}
V_{\lambda+\gamma+L}\otimes V_{\mu+\delta+L}\\
V_{\lambda+L}\otimes V_{\mu+L}\ V_{\gamma+L}\otimes V_{\delta+L}
\end{array}\right).
\]
Let $P_{\epsilon}$ be the projection of $V_{\lambda+\gamma+L}\otimes V_{\mu+\delta+L}$
to $\widetilde{\left(\lambda+\gamma\ \epsilon\right)}$ for $\epsilon=0,1.$
Then $P_{\epsilon}I$ is an nonzero intertwining operator in $I_{\mathcal{U}}\left(\begin{array}{c}
\widetilde{\left(\lambda+\gamma\ \epsilon\right)}\\
(\lambda\ \mu)\ (\gamma\ \delta)
\end{array}\right).$ So $N_{\mathcal{U}}\left(_{\left(\lambda\ \mu\right)\ \left(\gamma\ \delta\right)}^{\widetilde{\left(\lambda+\gamma\ \epsilon\right)}}\right)\geq1$
for $\epsilon=0,1.$ Since $(\lambda\ \mu)$ and $(\mu\ \lambda)$
are isomorphic $\mathcal{U}$-modules, $N_{\mathcal{U}}\left(_{\left(\lambda\ \mu\right)\ \left(\gamma\ \delta\right)}^{\widetilde{\left(\mu+\gamma\ \epsilon\right)}}\right)\geq1$
for $\epsilon=0,1.$

By Proposition \ref{quantum-product} and quantum
dimensions in Proposition \ref{quantum dimension} we see that
\[
q\dim_{\mathcal{U}}\left(\left(\lambda\ \mu\right)\boxtimes\left(\gamma\ \delta\right)\right)=4.
\] So we obtain (\ref{Fusion-Nondiag-Nodiag}).

\emph{Proof of (\ref{add}): }Now we have $V_{\lambda+L}\otimes V_{\mu+L}\cong\left(\lambda\ \mu\right),$
$V_{\gamma+L}\otimes V_{\delta+L}\cong\left(\gamma\ \delta\right)$,
$V_{\lambda+\gamma+L}\otimes V_{\mu+\delta+L}\cong\left(\lambda+\gamma\ \mu+\delta\right)$,
and $V_{\mu+\gamma+L}\otimes V_{\lambda+\delta+L}\cong\widetilde{\left(\mu+\gamma\ 0\right)}+\widetilde{\left(\mu+\gamma\ 1\right)}$
as $\mathcal{U}$-modules. Take $V=V_{L}\otimes V_{L}$ and $U=\mathcal{U}$
in Proposition 2.9 in \cite{ADL}, then (\ref{V_L-fusion-1}) implies
\[
1=N_{V_{L}\otimes V_{L}}\left(_{V_{\lambda+L}\otimes V_{\mu+L}\ V_{\gamma+L}\otimes V_{\delta+L}}^{V_{\lambda+\gamma+L}\otimes V_{\mu+\delta+L}}\right)\le N_{\mathcal{U}}\left(_{\left(\lambda\ \mu\right)\ \left(\gamma\ \delta\right)}^{\left(\lambda+\gamma\ \mu+\delta\right)}\right).
\]
So $N_{\mathcal{U}}\left(_{\left(\lambda\ \mu\right)\ \left(\gamma\ \delta\right)}^{\left(\lambda+\gamma\ \mu+\delta\right)}\right)\ge1$.

Using the condition $\mu+\gamma+L=\lambda+\delta+L$ and the proof of (\ref{Fusion-Nondiag-Nodiag}) gives
$N_{\mathcal{U}}\left(_{\left(\lambda\ \mu\right)\ \left(\gamma\ \delta\right)}^{\widetilde{\left(\mu+\gamma\ \epsilon\right)}}\right)\geq1$
for $\epsilon=0,1.$  Applying the formula
$q\dim_{\mathcal{U}}\left(\left(\lambda\ \mu\right)\boxtimes\left(\gamma\ \delta\right)\right)=4$
again to obtain (\ref{add}).

\emph{Proof of (\ref{Fusion-Nondiag-Nondiag-Case 2})}: Now we have
$V_{\lambda+L}\otimes V_{\mu+L}\cong\left(\lambda\ \mu\right)$, $V_{\gamma+L}\otimes V_{\delta+L}\cong\left(\gamma\ \delta\right)$,
$V_{\lambda+\mu+L}\otimes V_{\mu+\delta+L}\cong\left(\lambda+\gamma\ \mu+\delta\right)$
and $V_{\mu+\gamma+L}\otimes V_{\lambda+\delta+L}\cong\left(\mu+\gamma\ \lambda+\delta\right)$
as $\mathcal{U}$-modules.  From the proof of (\ref{add}) we see that
\[
1=N_{V_{L}\otimes V_{L}}\left(_{V_{\lambda+L}\otimes V_{\mu+L}\ V_{\gamma+L}\otimes V_{\delta+L}}^{V_{\lambda+\gamma+L}\otimes V_{\mu+\delta+L}}\right)\le N_{\mathcal{U}}\left(_{\left(\lambda\ \mu\right)\ \left(\gamma\ \delta\right)}^{\left(\lambda+\gamma\ \mu+\delta\right)}\right)
\]
and
\[
1=N_{V_{L}\otimes V_{L}}\left(_{V_{\mu+L}\otimes V_{\lambda+L}\ V_{\gamma+L}\otimes V_{\delta+L}}^{V_{\mu+\gamma+L}\otimes V_{\lambda+\delta+L}}\right)\le N_{\mathcal{U}}\left(_{\left(\mu\ \lambda\right)\ \left(\gamma\ \delta\right)}^{\left(\mu+\gamma\ \lambda+\delta\right)}\right).
\]

Notice that as irreducible $\mathcal{U}$-modules, $\left(\mu\ \lambda\right)\cong\left(\lambda\ \mu\right)$.
So $N_{\mathcal{U}}\left(_{\left(\lambda\ \mu\right)\ \left(\gamma\ \delta\right)}^{\left(\mu+\gamma\ \lambda+\delta\right)}\right)\ge1$. The result follows immediately by the fact that $q\dim_{\mathcal{U}}(\lambda+\gamma\ \mu+\delta)=q\dim_{\mathcal{U}}(\mu+\gamma\ \lambda+\delta)=2.$

\emph{Proof of} \emph{(\ref{Fusion-Nondiag-Twisted}):} First by Proposition
\ref{quantum-product} and quantum dimensions in Proposition \ref{quantum dimension}
\[
q\dim_{\mathcal{U}}\left(\left(\lambda\ \mu\right)\boxtimes\widehat{\left(\gamma\ \epsilon\right)}\right)=q\dim_{\mathcal{U}}\left(\lambda\ \mu\right)\cdot q\dim_{\mathcal{U}}\widehat{\left(\gamma\ \epsilon\right)}=2\sqrt{\left|L^{\circ}/L\right|}.
\]

By fusion rules in Proposition \ref{Fusion-V_L+}, for any $\alpha,\beta\in\mathcal{S}$,
$N_{V_{\sqrt{2}L}^{+}}\left(_{V_{\frac{\lambda-\mu+\alpha}{\sqrt{2}}+\sqrt{2}L}\ V_{\sqrt{2}L}^{T_{\gamma+\beta,\pm}}}^{\ \ \ W}\right)\not=0$
only if $W=V_{\sqrt{2}L}^{T_{\delta,}\pm}$ for some $\delta\in L^{\circ}$.
So $
N_{\mathcal{U}}\left(_{\left(\lambda\ \mu\right)\ \widehat{\left(\gamma\ 0\right)}}^{\left(\gamma_1\ \delta\right)}\right)=N_{\mathcal{U}}\left(_{\left(\lambda\ \mu\right)\ \widehat{\left(\gamma\ 0\right)}}^{\widetilde{\left(\gamma_2\ \epsilon\right)}}\right)=0$
for any $\gamma_1, \gamma_2, \delta\in L^{\circ}$ with $\gamma_1+L\not=\delta+L$ and $\epsilon=0,1.$
Using the classification of irreducible modules in Proposition \ref{all modules} we see that
 $\left(\lambda\ \mu\right)\boxtimes\widehat{\left(\gamma\ \epsilon\right)}=\widehat{\left(\delta_{1}\ \epsilon_{1}\right)}+\widehat{\left(\delta_{2}\ \epsilon_{2}\right)}$
for some $\delta_{1},\delta_{2}\in L^{\circ}$ and $\epsilon_{1},\epsilon_{2}\in\left\{ 0,1\right\} $.
So we only need to determine $\delta_{1},\delta_{1}$, $\epsilon_{1}$
and $\epsilon_{2}$.

We first determine $N_{\mathcal{U}}\left(_{\left(\lambda\ \mu\right)\ \widehat{\left(\gamma\ 0\right)}}^{\widehat{\left(\delta\ 0\right)}}\right)$ for $\delta\in L^{\circ}.$
From Proposition 2.9 in \cite{ADL} with $V=\mathcal{U}$, $U=\mathcal{V},$
we have
\[
N_{\mathcal{U}}\left(_{\left(\lambda\ \mu\right)\ \widehat{\left(\lambda_{1}\ 0\right)}}^{\widehat{\left(\delta\ 0\right)}}\right)\le N_{\mathcal{V}}\left(_{V_{\frac{\lambda+\mu}{\sqrt{2}}+\sqrt{2}L}\otimes V_{\frac{\lambda-\mu}{\sqrt{2}}+\sqrt{2}L\ }V_{\frac{\gamma}{\sqrt{2}}+\sqrt{2}L}\otimes V_{\sqrt{2}L}^{T_{\chi_{\gamma}},+}}^{\widehat{\left(\delta\ 0\right)}}\right).
\]

From Proposition \ref{all modules} we have
\begin{align*}
 & N_{\mathcal{V}}\left(_{V_{\frac{\lambda+\mu}{\sqrt{2}}+\sqrt{2}L}\otimes V_{\frac{\lambda-\mu}{\sqrt{2}}+\sqrt{2}L\ }V_{\frac{\gamma}{\sqrt{2}}+\sqrt{2}L}\otimes V_{\sqrt{2}L}^{T_{\chi_{\gamma}},+}}^{\widehat{\left(\delta\ 0\right)}}\right)\\
 & \leq \sum_{\alpha\in\mathcal{S},\chi_{\delta}\left(\sqrt{2}\alpha\right)=1}N_{\mathcal{V}}\left(_{V_{\frac{\lambda+\mu}{\sqrt{2}}+\sqrt{2}L}\otimes V_{\frac{\lambda-\mu}{\sqrt{2}}+\sqrt{2}L\ }V_{\frac{\gamma}{\sqrt{2}}+\sqrt{2}L}\otimes V_{\sqrt{2}L}^{T_{\chi_{\gamma}},+}}^{V_{\frac{\delta+\alpha}{\sqrt{2}}+\sqrt{2}L}\otimes V_{\sqrt{2}L}^{T_{\chi_{\delta+\alpha}},+}}\right)\\
 & +\sum_{\alpha\in\mathcal{S},\chi_{\delta}\left(\sqrt{2}\alpha\right)=-1}N_{\mathcal{V}}\left(_{V_{\frac{\lambda+\mu}{\sqrt{2}}+\sqrt{2}L}\otimes V_{\frac{\lambda-\mu}{\sqrt{2}}+\sqrt{2}L\ }V_{\frac{\gamma}{\sqrt{2}}+\sqrt{2}L}\otimes V_{\sqrt{2}L}^{T_{\chi_{\gamma}},+}}^{V_{\frac{\delta+\alpha}{\sqrt{2}}+\sqrt{2}L}\otimes V_{\sqrt{2}L}^{T_{\chi_{\delta+\alpha}},-}}\right)\\
 & =\sum_{\alpha\in\mathcal{S},\chi_{\delta}\left(\sqrt{2}\alpha\right)=1}N_{V_{\sqrt{2}L}}\left(_{V_{\frac{\lambda+\mu}{\sqrt{2}}+\sqrt{2}L}\ V_{\frac{\gamma}{\sqrt{2}}+\sqrt{2}L}}^{V_{\frac{\delta+\alpha}{\sqrt{2}}+\sqrt{2}L}}\right)\cdot N_{V_{\sqrt{2}L}^{+}}\left(_{V_{\frac{\lambda-\mu}{\sqrt{2}}+\sqrt{2}L\ }V_{\sqrt{2}L}^{T_{\chi_{\gamma}},+}}^{V_{\sqrt{2}L}^{T_{\chi_{\delta+\alpha}},+}}\right)\\
 & +\sum_{\alpha\in\mathcal{S},\chi_{\delta}\left(\sqrt{2}\alpha\right)=-1}N_{V_{\sqrt{2}L}}\left(_{V_{\frac{\lambda+\mu}{\sqrt{2}}+\sqrt{2}L}\ V_{\frac{\gamma}{\sqrt{2}}+\sqrt{2}L}}^{V_{\frac{\delta+\alpha}{\sqrt{2}}+\sqrt{2}L}}\right)\cdot N_{V_{\sqrt{2}L}^{+}}\left(_{V_{\frac{\lambda-\mu}{\sqrt{2}}+\sqrt{2}L\ }V_{\sqrt{2}L}^{T_{\chi_{\gamma}},+}}^{V_{\sqrt{2}L}^{T_{\chi_{\delta+\alpha}},-}}\right).
\end{align*}

It is clear from the fusion rules for vertex operator algebra $V_{\sqrt{2}L}$ (see Proposition \ref{Fusion-V_L}) that if $\delta+L\ne \lambda+\mu+\gamma +L$ then
\[
N_{\mathcal{V}}\left(_{V_{\frac{\lambda+\mu}{\sqrt{2}}+\sqrt{2}L}\otimes V_{\frac{\lambda-\mu}{\sqrt{2}}+\sqrt{2}L\ }V_{\frac{\gamma}{\sqrt{2}}+\sqrt{2}L}\otimes V_{\sqrt{2}L}^{T_{\chi_{\gamma}},+}}^{\widehat{\left(\delta\ 0\right)}}\right)=0.
\]
We now assume that $\delta=\lambda+\mu+\gamma.$ For $\alpha\in\mathcal{S}$ with $\chi_{\delta}\left(\sqrt{2}\alpha\right)=1$,
\[
N_{V_{\sqrt{2}L}}\left(_{V_{\frac{\lambda+\mu}{\sqrt{2}}+\sqrt{2}L}\ V_{\frac{\gamma}{\sqrt{2}}+\sqrt{2}L}}^{V_{\frac{\lambda+\mu+\gamma+\alpha}{\sqrt{2}}+\sqrt{2}L}}\right)=N_{V_{\sqrt{2}L}^{+}}\left(_{V_{\frac{\lambda-\mu}{\sqrt{2}}+\sqrt{2}L\ }V_{\sqrt{2}L}^{T_{\chi_{\gamma}},+}}^{V_{\sqrt{2}L}^{T_{\chi_{\lambda+\mu+\gamma+\alpha}},+}}\right)=1
\]
only if $\alpha\in 2L$ and $\chi_{\lambda+\mu+\gamma+\alpha}=\chi_{\gamma}^{\left(\frac{\lambda-\mu}{\sqrt{2}}\right)}$.  Note that  $\alpha\in {\cal S}$
lies in $2L$ if and only if $\alpha=0.$
Since
\[
\chi_{\lambda+\mu+\gamma}\left(\sqrt{2}\alpha_{i}\right)=\left(-1\right)^{\frac{\left\langle \alpha_{i},\alpha_{i}\right\rangle }{2}+\left\langle \lambda+\mu+\gamma,\alpha_{i}\right\rangle },
\]
and
\[
\chi_{\gamma}^{\left(\frac{\lambda-\mu}{\sqrt{2}}\right)}\left(\sqrt{2}\alpha_{i}\right)=\left(-1\right)^{\left\langle \lambda-\mu,\alpha_{i}\right\rangle }\chi_{\gamma}\left(\sqrt{2}\alpha_{i}\right)=\left(-1\right)^{\frac{\left\langle \alpha_{i},\alpha_{i}\right\rangle }{2}+\left\langle \lambda-\mu+\gamma,\alpha_{i}\right\rangle },
\]
we see that  $\chi_{\lambda+\mu+\gamma}=\chi_{\gamma}^{\left(\frac{\lambda-\mu}{\sqrt{2}}\right)}.$
So
\[
\sum_{\alpha\in\mathcal{S},\chi_{\delta}\left(\sqrt{2}\alpha\right)=1}N_{V_{\sqrt{2}L}}\left(_{V_{\frac{\lambda+\mu}{\sqrt{2}}+\sqrt{2}L}\ V_{\frac{\gamma}{\sqrt{2}}+\sqrt{2}L}}^{V_{\frac{\delta+\alpha}{\sqrt{2}}+\sqrt{2}L}}\right)\cdot N_{V_{\sqrt{2}L}^{+}}\left(_{V_{\frac{\lambda-\mu}{\sqrt{2}}+\sqrt{2}L\ }V_{\sqrt{2}L}^{T_{\chi_{\gamma}},+}}^{V_{\sqrt{2}L}^{T_{\chi_{\delta+\alpha}},+}}\right)=1.
\]

For $\alpha\in\mathcal{S}$ with $\chi_{\delta}\left(\sqrt{2}\alpha\right)=-1$, we clearly have
\[
N_{V_{\sqrt{2}L}}\left(_{V_{\frac{\lambda+\mu}{\sqrt{2}}+\sqrt{2}L}\ V_{\frac{\gamma}{\sqrt{2}}+\sqrt{2}L}}^{V_{\frac{\delta+\alpha}{\sqrt{2}}+\sqrt{2}L}}\right)=0
\]
and consequently
\[
\sum_{\alpha\in\mathcal{S},\chi_{\delta}\left(\sqrt{2}\alpha\right)=-1}N_{V_{\sqrt{2}L}}\left(_{V_{\frac{\lambda+\mu}{\sqrt{2}}+\sqrt{2}L}\ V_{\frac{\gamma}{\sqrt{2}}+\sqrt{2}L}}^{V_{\frac{\delta+\alpha}{\sqrt{2}}+\sqrt{2}L}}\right)\cdot N_{V_{\sqrt{2}L}^{+}}\left(_{V_{\frac{\lambda-\mu}{\sqrt{2}}+\sqrt{2}L\ }V_{\sqrt{2}L}^{T_{\chi_{\gamma}},+}}^{V_{\sqrt{2}L}^{T_{\chi_{\delta+\alpha}},-}}\right)=0.
\]
Thus
\[
N_{\mathcal{U}}\left(_{\left(\lambda\ \mu\right)\ \widehat{\left(\gamma\ 0\right)}}^{\widehat{\left(\lambda+\mu+\gamma\ 0\right)}}\right)\le1\ \mbox{and\ }N_{\mathcal{U}}\left(_{\left(\lambda\ \mu\right)\ \widehat{\left(\gamma\ 0\right)}}^{\widehat{\left(\delta\ 0\right)}}\right)=0\ \mbox{if}\ \delta+L\not=\lambda+\mu+\gamma+L.
\]

Similarly, $N_{\mathcal{U}}\left(_{\left(\lambda\ \mu\right)\ \widehat{\left(\gamma\ 0\right)}}^{\widehat{\left(\lambda+\mu+\gamma\ 1\right)}}\right)\le1$
and $N_{\mathcal{U}}\left(_{\left(\lambda\ \mu\right)\ \widehat{\left(\gamma\ 0\right)}}^{\widehat{\left(\delta\ 1\right)}}\right)=0$
if $\delta+L\not=\lambda+\mu+\gamma+L.$ By quantum dimensions in Proposition \ref{quantum dimension}, we
get
\[
N_{\mathcal{U}}\left(_{\left(\lambda\ \mu\right)\ \widehat{\left(\gamma\ 0\right)}}^{\widehat{\left(\lambda+\mu+\gamma\ 0\right)}}\right)=N_{\mathcal{U}}\left(_{\left(\lambda\ \mu\right)\ \widehat{\left(\gamma\ 0\right)}}^{\widehat{\left(\lambda+\mu+\gamma\ 1\right)}}\right)=1,
\]
as expected.

\emph{Proof of (\ref{Fusion Prod-dia-twisted}):} First by Proposition
\ref{quantum-product} and quantum dimensions in Proposition \ref{quantum dimension}
\[
q\dim_{\mathcal{U}}\left(\widetilde{\left(\lambda\ \epsilon\right)}\boxtimes\widehat{\left(\mu\ \epsilon_{1}\right)}\right)=q\dim_{\mathcal{U}}\widetilde{\left(\lambda\ \epsilon\right)}\cdot q\dim_{\mathcal{U}}\widehat{\left(\mu\ \epsilon_{1}\right)}=\sqrt{\left|L^{\circ}/L\right|}.
\]
By fusion rules in Proposition \ref{Fusion-V_L+}, for any $\alpha,\beta\in\mathcal{S}$,
$N_{V_{\sqrt{2}L}^{+}}\left(_{V_{\frac{\alpha}{\sqrt{2}}+\sqrt{2}L}^{+}\ V_{\sqrt{2}L}^{T_{\mu+\beta,\pm}}}^{\ \ W}\right)\not=0$
only if $W=V_{\sqrt{2}L}^{T_{\gamma,}\pm}$ for some $\gamma\in L^{\circ}$.
So $N_{\mathcal{U}}\left(_{\widetilde{\left(\lambda\ \epsilon\right)}\ \widehat{\left(\mu\ \epsilon_{1}\right)}}^{\left(\gamma\ \delta\right)}\right)=N_{\mathcal{U}}\left(_{\widetilde{\left(\lambda\ \epsilon\right)}\ \widehat{\left(\mu\ \epsilon_{1}\right)}}^{\widetilde{\left(\gamma_1\ \epsilon\right)}}\right)=0$
for any $\gamma,\delta, \gamma_1\in L^{\circ}$ with $\gamma+L\not=\delta+L$ and $\epsilon, \epsilon_1=0,1.$
As a result, $\widetilde{\left(\lambda\ \epsilon\right)}\boxtimes\widehat{\left(\mu\ \epsilon_{1}\right)}=\widehat{\left(\delta\ \epsilon_{2}\right)}$
for some $\delta\in L^{\circ},\epsilon_{1}, \epsilon_{2}\in\left\{ 0,1\right\}.$

First we consider $N_{\mathcal{U}}\left(_{\widetilde{\left(\lambda\ 0\right)}\ \widehat{\left(\mu\ 0\right)}}^{\widehat{\left(\delta\ 0\right)}}\right)$. The proof of (\ref{Fusion-Nondiag-Twisted}) shows that
$N_{\mathcal{U}}\left(_{\widetilde{\left(\lambda\ 0\right)}\ \widehat{\left(\mu\ 0\right)}}^{\widehat{\left(2\lambda+\mu\ 0\right)}}\right)\le1$ and $N_{\mathcal{U}}\left(_{\widetilde{\left(\lambda\ 0\right)}\ \widehat{\left(\mu\ 0\right)}}^{\widehat{\left(\delta\ 0\right)}}\right)=0$ if $\delta+L\not=2\lambda+\mu+L.$

Now we consider $N_{\mathcal{U}}\left(_{\widetilde{\left(\lambda\ 0\right)}\ \widehat{\left(\mu\ 0\right)}}^{\widehat{\left(\delta\ 1\right)}}\right)$. We have
\[
N_{\mathcal{U}}\left(_{\widetilde{\left(\lambda\ 0\right)}\ \widehat{\left(\mu\ 0\right)}}^{\widehat{\left(\delta\ 1\right)}}\right)\le N_{\mathcal{V}}\left(_{V_{\sqrt{2}\lambda+\sqrt{2}L}\otimes V_{\sqrt{2}L\ }^{+}V_{\frac{\mu}{\sqrt{2}}+\sqrt{2}L}\otimes V_{\sqrt{2}L}^{T_{\chi_{\mu}},+}}^{\widehat{\left(\delta\ 1\right)}}\right).
\]

 Using $\widehat{\left(\delta\ 1\right)}=\sum_{\alpha\in\mathcal{S},\chi_{\delta}\left(\sqrt{2}\alpha\right)=1}V_{\frac{\delta+\alpha}{\sqrt{2}}+\sqrt{2}L}\otimes V_{\sqrt{2}L}^{T_{\chi_{\delta+\alpha}},-}+\sum_{\alpha\in\mathcal{S},\chi_{\delta}\left(\sqrt{2}\alpha\right)=-1}V_{\frac{\delta+\alpha}{\sqrt{2}}+\sqrt{2}L}\otimes V_{\sqrt{2}L}^{T_{\chi_{\delta+\alpha}},+}$ gives
\begin{align*}
 & N_{\mathcal{V}}\left(_{V_{\sqrt{2}\lambda+\sqrt{2}L}\otimes V_{\sqrt{2}L\ }^{+}V_{\frac{\mu}{\sqrt{2}}+\sqrt{2}L}\otimes V_{\sqrt{2}L}^{T_{\chi_{\mu}},+}}^{\widehat{\left(\delta\ 1\right)}}\right)\\
 & \leq \sum_{\alpha\in\mathcal{S},\chi_{\delta}\left(\sqrt{2}\alpha\right)=1}N_{\mathcal{V}}\left(_{V_{\sqrt{2}\lambda+\sqrt{2}L}\otimes V_{\sqrt{2}L\ }^{+}V_{\frac{\mu}{\sqrt{2}}+\sqrt{2}L}\otimes V_{\sqrt{2}L}^{T_{\chi_{\mu}},+}}^{V_{\frac{\delta+\alpha}{\sqrt{2}}+\sqrt{2}L}\otimes V_{\sqrt{2}L}^{T_{\chi_{\delta+\alpha}},-}}\right)\\
 & +\sum_{\alpha\in\mathcal{S},\chi_{\delta}\left(\sqrt{2}\alpha\right)=-1}N_{\mathcal{V}}\left(_{V_{\sqrt{2}\lambda+\sqrt{2}L}\otimes V_{\sqrt{2}L\ }^{+}V_{\frac{\mu}{\sqrt{2}}+\sqrt{2}L}\otimes V_{\sqrt{2}L}^{T_{\chi_{\mu}},+}}^{V_{\frac{\delta+\alpha}{\sqrt{2}}+\sqrt{2}L}\otimes V_{\sqrt{2}L}^{T_{\chi_{\delta+\alpha}},+}}\right)\\
 & =\sum_{\alpha\in\mathcal{S},\chi_{\delta}\left(\sqrt{2}\alpha\right)=1}N_{V_{\sqrt{2}L}}\left(_{V_{\sqrt{2}\lambda+\sqrt{2}L}\ V_{\frac{\mu}{\sqrt{2}}+\sqrt{2}L}}^{V_{\frac{\delta+\alpha}{\sqrt{2}}+\sqrt{2}L}}\right)\cdot N_{V_{\sqrt{2}L}^{+}}\left(_{V_{\sqrt{2}L\ }^{+}V_{\sqrt{2}L}^{T_{\chi_{\mu}},+}}^{V_{\sqrt{2}L}^{T_{\chi_{\delta+\alpha}},-}}\right)\\
 & +\sum_{\alpha\in\mathcal{S},\chi_{\delta}\left(\sqrt{2}\alpha\right)=-1}N_{V_{\sqrt{2}L}}\left(_{V_{\sqrt{2}\lambda+\sqrt{2}L}\ V_{\frac{\mu}{\sqrt{2}}+\sqrt{2}L}}^{V_{\frac{\delta+\alpha}{\sqrt{2}}+\sqrt{2}L}}\right)\cdot N_{V_{\sqrt{2}L}^{+}}\left(_{V_{\sqrt{2}L\ }^{+}V_{\sqrt{2}L}^{T_{\chi_{\mu}},+}}^{V_{\sqrt{2}L}^{T_{\chi_{\delta+\alpha}},+}}\right).
\end{align*}

Clearly, $N_{\mathcal{V}}\left(_{V_{\sqrt{2}\lambda+\sqrt{2}L}\otimes V_{\sqrt{2}L\ }^{+}V_{\frac{\mu}{\sqrt{2}}+\sqrt{2}L}\otimes V_{\sqrt{2}L}^{T_{\chi_{\mu}},+}}^{\widehat{\left(\delta\ 1\right)}}\right)=0$ if $\delta+L\ne 2\lambda+\mu+L.$ So we can assume that $\delta=2\mu+\lambda.$
Then for any $\alpha\ne 0$ we have $N_{V_{\sqrt{2}L}}\left(_{V_{\sqrt{2}\lambda+\sqrt{2}L}\ V_{\frac{\mu}{\sqrt{2}}+\sqrt{2}L}}^{V_{\frac{\delta+\alpha}{\sqrt{2}}+\sqrt{2}L}}\right)=0.$
So
\begin{align*}
& N_{\mathcal{V}}\left(_{V_{\sqrt{2}\lambda+\sqrt{2}L}\otimes V_{\sqrt{2}L\ }^{+}V_{\frac{\mu}{\sqrt{2}}+\sqrt{2}L}\otimes V_{\sqrt{2}L}^{T_{\chi_{\mu}},+}}^{\widehat{\left(2\lambda+\mu 1\right)}}\right)\\
& \leq
  N_{V_{\sqrt{2}L}}\left(_{V_{\sqrt{2}\lambda+\sqrt{2}L}\ V_{\frac{\mu}{\sqrt{2}}+\sqrt{2}L}}^{V_{\frac{2\lambda+\mu}{\sqrt{2}}+\sqrt{2}L}}\right)\cdot N_{V_{\sqrt{2}L}^{+}}\left(_{V_{\sqrt{2}L\ }^{+}V_{\sqrt{2}L}^{T_{\chi_{\mu}},+}}^{V_{\sqrt{2}L}^{T_{\chi_{2\lambda+\mu}},-}}\right).
  \end{align*}
By fusion rules in Proposition \ref{Fusion-V_L+}, $N_{V_{\sqrt{2}L}^{+}}\left(_{V_{\sqrt{2}L\ }^{+}V_{\sqrt{2}L}^{T_{\chi_{\mu}},+}}^{V_{\sqrt{2}L}^{T_{\chi_{\delta+\alpha}},-}}\right)=0$
for any $\delta\in L^{\circ}$ and $\alpha\in\mathcal{S}$.
 Thus $N_{\mathcal{V}}\left(_{V_{\sqrt{2}\lambda+\sqrt{2}L}\otimes V_{\sqrt{2}L\ }^{+}V_{\frac{\mu}{\sqrt{2}}+\sqrt{2}L}\otimes V_{\sqrt{2}L}^{T_{\chi_{\mu}},+}}^{\widehat{\left(\delta\ 1\right)}}\right)=0$
for any $\delta\in L^{\circ}$. By counting quantum dimensions, we
obtain
\[
\widetilde{\left(\lambda\ 0\right)}\boxtimes\widehat{\left(\mu\ 0\right)}=\widehat{\left(2\lambda+\mu\ 0\right).}
\]

Use similar argument, we can prove for $\epsilon,\epsilon_{1}=0,1$,
$\widetilde{\left(\lambda\ \epsilon\right)}\boxtimes\widehat{\left(\mu\ \epsilon_{1}\right)}=\widehat{\left(2\lambda+\mu\ \epsilon+\epsilon_{1}\right)}.$

\emph{Proof of (\ref{Fusion-Twisted-Twisted-Case 1}):} Recall that $\mathcal{T}$
is a complete set of representatives of $L$ in $L^{\circ}$ and $\left|\mathcal{T}\right|=l.$ If $\frac{\lambda+\mu}{2}\in L^{\circ}$, by
Proposition \ref{fusion rule symmmetry property}, Remark \ref{Dual of U-modules},
and fusion product (\ref{Fusion Prod-dia-twisted}), we see that $N_{\mathcal{U}}\left(_{\widehat{\left(\lambda\ \epsilon\right)}\ \widehat{\left(\mu\ \epsilon_{1}\right)}}^{\widetilde{\left(\frac{\lambda+\mu}{2}\ \epsilon-\epsilon_{1}\right)}}\right)=1.$
 Note that $\widehat{\left(\lambda+\alpha\ \epsilon\right)}\cong\widehat{\left(\lambda\ \epsilon\right)}$
for $\lambda\in\mathcal{T},\alpha\in L$ and $\epsilon=0,1$. Then by Proposition
\ref{fusion rule symmmetry property}, Remark \ref{Dual of U-modules},
and fusion product (\ref{Fusion Prod-dia-twisted}), we also have
$N_{\mathcal{U}}\left(_{\widehat{\left(\lambda+2\gamma\ \epsilon\right)}\ \widehat{\left(\mu\ \epsilon_{1}\right)}}^{\widetilde{\left(\frac{\lambda+\mu}{2}+\gamma\ \epsilon-\epsilon_{1}\right)}}\right)=1$
for all $\gamma\in G$. Clearly,  for $\gamma_{1},\gamma_{2}\in G$
with $\gamma_{1}\not=\gamma_{2}$, $\frac{\lambda+\mu}{2}+\gamma_{1}+L\not=\frac{\lambda+\mu}{2}+\gamma_{2}+L$.
So $\widehat{\left(\frac{\lambda+\mu}{2}+\gamma_{1}\ \epsilon-\epsilon_{1}\right)}$ and $\widehat{\left(\frac{\lambda+\mu}{2}+\gamma_{2}\ \epsilon-\epsilon_{1}\right)}$
are not isomorphic ${\cal U}$-modules. Also notice that for any $\gamma\in G$, $\widehat{\left(\lambda+2\gamma\ \epsilon\right)}\cong\widehat{\left(\lambda\ \epsilon\right)}$.
Assume $\left|G\right|=n$, then the number of isomorphism classes
of $\widetilde{\left(\frac{\lambda+\mu}{2}+\gamma\ \epsilon-\epsilon_{1}\right)}$
such that $N_{\mathcal{U}}\left(_{\widehat{\left(\lambda\ \epsilon\right)}\ \widehat{\left(\mu\ \epsilon_{1}\right)}}^{\widetilde{\left(\frac{\lambda+\mu}{2}+\gamma\ \epsilon-\epsilon_{1}\right)}}\right)=1$
is equal to $n$.

By Proposition \ref{fusion rule symmmetry property}, Remark \ref{Dual of U-modules},
and fusion product (\ref{Fusion-Nondiag-Twisted}), we see that $N_{\mathcal{U}}\left(_{\widehat{\left(\lambda\ \epsilon\right)}\ \widehat{\left(\mu\ \epsilon_{1}\right)}}^{\left(\lambda+\mu-\delta\ \delta\right)}\right)=1$
for any $\delta\in\mathcal{T}$ such that $\delta+L\not=\lambda+\mu-\delta+L$.
That is, $N_{\mathcal{U}}\left(_{\widehat{\left(\lambda\ \epsilon\right)}\ \widehat{\left(\mu\ \epsilon_{1}\right)}}^{\left(\mu+\lambda-\delta\ \delta\right)}\right)=1$
for all $\delta\in\mathcal{T}$ satisfying that $\delta$ cannot be
written of the from $\frac{\lambda+\mu}{2}+\gamma$ with $\gamma\in G$.
The number of such $\delta\in\mathcal{T}$ is $l-n$. Also note that
$\left(\lambda\ \mu\right)\cong\left(\mu\ \lambda\right)$ as irreducible
$\mathcal{U}$-modules for any $\lambda,\mu\in\mathcal{T}$. Thus
the number of isomorphism classes of $\left(\lambda+\mu-\delta\ \delta\right)$
such that $N_{\mathcal{U}}\left(_{\widehat{\left(\lambda\ \epsilon\right)}\ \widehat{\left(\mu\ \epsilon_{1}\right)}}^{\left(\lambda+\mu-\delta\ \delta\right)}\right)=1$
is $\frac{l-n}{2}$. By counting quantum dimensions, we see that
\[
\widehat{\left(\lambda\ \epsilon\right)}\boxtimes\widehat{\left(\mu\ \epsilon_{1}\right)}=\sum_{\gamma\in G}\widetilde{\left(\frac{\lambda+\mu}{2}+\gamma\ \epsilon-\epsilon_{1}\right)}+\sum_{\delta\in\mathcal{T},\delta\not=\frac{\lambda+\mu}{2}+\gamma,\gamma\in G}\left(\lambda+\mu-\delta\ \delta\right).
\]

\emph{Proof of (\ref{Fusion-Twisted-Twisted-Case 2}):} By Proposition
\ref{fusion rule symmmetry property}, Remark \ref{Dual of U-modules},
and fusion product (\ref{Fusion-Nondiag-Twisted}), we see that
\[
N_{\mathcal{U}}\left(_{\widehat{\left(\lambda\ \epsilon\right)}\ \widehat{\left(\mu\ \epsilon_{1}\right)}}^{\left(\lambda+\mu-\delta\ \delta\right)}\right)=1\ \mbox{for\ any}\ \delta\in\mbox{\ensuremath{\mathcal{T\ }}such\ }\mbox{that\ }\delta+L\not=\lambda+\mu-\delta+L.
\]
 Since $\frac{\lambda+\mu}{2}\not\in L^{\circ}$, we see that every
$\delta\in\mathcal{T}$ satisfy such condition. Thus the number of
such $\delta$ is equal to $l$. Notice that $\left(\lambda\ \mu\right)\cong\left(\mu\ \lambda\right)$
as irreducible $\mathcal{U}$-modules for any $\lambda,\mu\in\mathcal{T}.$
Thus the number of isomorphism classes of $\left(\lambda+\mu-\delta\ \delta\right)$
such that $N_{\mathcal{U}}\left(_{\widehat{\left(\lambda\ \epsilon\right)}\ \widehat{\left(\mu\ \epsilon_{1}\right)}}^{\left(\lambda+\mu-\delta\ \delta\right)}\right)=1$
is $\frac{l}{2}$. Now the quantum dimension of $\sum_{\delta\in\mathcal{T},\delta+L\not\not=\lambda+\mu-\delta+L}\left(\lambda+\mu-\delta\ \delta\right)$
is $l$ and the proof of (\ref{Fusion-Twisted-Twisted-Case 2}) is complete.

\end{proof}

\end{document}